\documentclass{eptcs}

\usepackage{amsmath,amsthm,amssymb,xspace,stmaryrd,microtype,IEEEtrantools,tikz}

\usetikzlibrary{positioning,calc,matrix,arrows}
\tikzstyle{mor}=[execute at begin node=$\scriptstyle, execute at end node=$]
\tikzstyle{commutative diagram} = [text height=1.5ex, text depth=.25ex,
    row sep=1cm, column sep=1cm]

\newcommand{\overbar}[1]{\mkern 1.5mu\overline{\mkern-1.5mu#1\mkern-1.5mu}\mkern 1.5mu}
\newcommand{\positive}{^+}

\newcommand{\eaplus}{\mathop\boxplus}
\newcommand{\coproduct}{\text{\rotatebox[origin=c]{180}{$\prod$}}}

\newcommand{\Zord}[1]{\mathcal{Z}_{\leq}^{#1}}
\newcommand{\Hord}[1]{\mathrm{H}_{\leq}^{#1}}
\newcommand{\CHord}[1]{\mathrm{H}_{\leq}^{#1}}

\newcommand{\Pow}{\mathcal{P}}
\newcommand{\Yoneda}{\mathcal{Y}}

\newcommand{\cycle}[1]{C_{#1}}
\newcommand{\tests}{T}
\newcommand{\circlegroup}{\ensuremath{\mathbb{S}^1}}

\newcommand{\Cat}[1]{\ensuremath{\mathbf{#1}}}

\newcommand{\Sets}{\Cat{Sets}\xspace}

\newcommand{\cSets}{\Cat{cSets}\xspace}

\newcommand{\EA}{\Cat{EA}\xspace}

\newcommand{\Cycl}{\Lambda\xspace}

\newcommand{\op}{^{\mathrm{op}}} 

\newcommand{\Proj}{\mathcal{P}\!roj}

\DeclareMathOperator{\Hom}{Hom}
\DeclareMathOperator{\St}{St}
\DeclareMathOperator{\im}{im}
\DeclareMathOperator{\HC}{HC}
\DeclareMathOperator{\HH}{HH}
\DeclareMathOperator{\Disc}{Disc}
\DeclareMathOperator{\precone}{prec}

\newtheorem{theorem}{Theorem}[section]
\newtheorem{lemma}[theorem]{Lemma}
\newtheorem{corollary}[theorem]{Corollary}
\newtheorem{proposition}[theorem]{Proposition}
\theoremstyle{definition}
\newtheorem{definition}[theorem]{Definition}
\newtheorem{example}[theorem]{Example}
\newtheorem{examples}[theorem]{Examples}
\theoremstyle{remark}

\usepackage{hyperref}

\title{Cohomology of Effect Algebras}
\author{Frank Roumen
    \institute{Computer Laboratory, University of Cambridge \\
        15 JJ Thomson Avenue, Cambridge CB3 0FD, United Kingdom}
    \email{fr338@cl.cam.ac.uk}}

\def\titlerunning{Cohomology of Effect Algebras}
\def\authorrunning{Frank Roumen}

\begin{document}

\maketitle

\begin{abstract}
    We will define two ways to assign cohomology groups to effect algebras,
    which occur in the algebraic study of quantum logic. The first way is
    based on Connes' cyclic cohomology. The resulting cohomology groups are
    related to the state space of the effect algebra, and can be computed
    using variations on the K\"unneth and Mayer--Vietoris sequences. The
    second way involves a chain complex of ordered abelian groups, and gives
    rise to a cohomological characterization of state extensions on effect
    algebras. This has applications to no-go theorems in quantum foundations,
    such as Bell's theorem.
\end{abstract}

\section{Introduction}
\label{sec:Intro}

Cohomology groups can be assigned to various mathematical structures, such as
topological spaces or groups, and are frequently helpful to classify certain
properties of the structure. For example, the cohomology groups of a
topological space provide information about the holes in the space, and the
second cohomology group of a group classifies its extensions. The main purpose
of this article is to define cohomology of effect algebras, and study its
applications to no-go theorems in quantum foundations.

Effect algebras were introduced in \cite{FoulisB94} as an abstract framework
for unsharp measurements in quantum mechanics. There are two reasons why their
cohomology may be applicable to no-go theorems. Firstly, as shown in
\cite{AbramskyMS11} based on earlier work in \cite{AbramskyB11}, sheaf
cohomology of measurement covers has proven to be fruitful in the
investigation of locality and non-contextuality. Measurement covers are
loosely related to effect algebras via the framework of test spaces.
Therefore one expects that the techniques used in \cite{AbramskyMS11} have
analogues in the world of effect algebras.  Secondly, in \cite{StatonU15} it
has been shown that the Bell paradox can be formulated in terms of
(non)-existence of factorizations in the category of effect algebras. Since
cohomology is often used to determine whether factorizations exist, a
cohomology theory of effect algebras will allow us to examine Bell's theorem
in a new way.

We will propose two different cohomology theories for effect algebras. The
first definition of cohomology is based on Connes' cyclic cohomology from
\cite{Connes83}. This is inspired by the close connection between effect
algebras and the abstract circles from \cite{Moerdijk96}. Since abstract
circles are relevant for cyclic cohomology, it is natural to consider cyclic
cohomology of effect algebras.

Most cohomology theories are obtained by assigning a sequence of abelian
groups to a mathematical object. Since effect algebras are ordered structures,
it will turn out to be productive to use a sequence of ordered abelian groups
instead. This will lead to the second definition of cohomology, which we call
order cohomology. It is loosely related to Pulmannov\'a's classification of
extensions of certain ordered algebraic structures in \cite{Pulmannova06}.

Both approaches for defining cohomology have advantages and disadvantages.
Cyclic cohomology is more suited for theoretical investigations, since it
opens up the possibility of using the powerful techniques from homological
algebra. For example, we will show how cyclic cohomology interacts with
products, coproducts, intersections, and unions of effect algebras. For order
cohomology, it is less clear what the interactions are, due to a lack of
general theory of homological algebra for ordered abelian groups. On the other
hand, order cohomology lends itself better to applications to quantum
mechanical no-go theorems. We will provide cohomological characterizations for
when a state on a certain probabilistic system is classically realizable, for
both cyclic and order cohomology. In the cyclic case, we only obtain a
necessary condition for realizability, so in certain scenarios false positives
may arise. A similar phenomenon occurs in the cohomological analysis of
contextuality in \cite{AbramskyMS11}. Order cohomology repairs this defect of
cyclic cohomology, since the order allows us to obtain a necessary and
sufficient condition for realizability of states.

The outline of this paper is as follows. Section~\ref{sec:EffectAlgebras}
contains preliminaries on effect algebras. Since the chain complexes
associated to an effect algebra are all based on the tests on the algebra,
Section~\ref{sec:Tests} continues the preliminaries by considering the
interplay between an effect algebra and its tests. In particular we define
cyclic sets and show that the tests on an effect algebra form a cyclic set.
Sections~\ref{sec:CyclicCohomologyEffectAlgebra} through
\ref{sec:Applications} are all about cyclic cohomology.
Section~\ref{sec:CyclicCohomologyEffectAlgebra} defines cyclic cohomology of
effect algebras and gives interpretations of the cyclic cohomology groups in
degrees 0 and 1. In particular, we show that the first cohomology group is
closely related to the state space of the effect algebra. Then we will turn to
computational techniques for determining cyclic cohomology groups. Often these
involve relative cohomology groups, defined in
Section~\ref{sec:RelativeCohomology}. As an application, we will prove that
cyclic cohomology preserves coproducts. The cohomology of a product can be
computed using the K\"unneth sequence (Section~\ref{sec:Kunneth}), and the
Mayer--Vietoris sequence is helpful for determining the cohomology of a union
of subalgebras (Section~\ref{sec:MayerVietoris} and
\ref{sec:GeneralizedMayerVietoris}).
Section~\ref{sec:Applications} treats applications to no-go theorems.
Finally, in Section~\ref{sec:OrderCohomology}, we switch to order cohomology.
There we will present various facts about ordered abelian groups, culminating
in a definition of cohomology of a chain complex of ordered abelian groups and
applications to effect algebras and no-go theorems.

\section{Effect algebras}
\label{sec:EffectAlgebras}

There are several proposals about what the right logic of quantum mechanics
should be. The first attempt to formalize quantum logic was made by Birkhoff
and von Neumann in \cite{BirkhoffN36}, based on orthomodular lattices. Effect
algebras generalize orthomodular lattices by also incorporating the
probabilistic structure of quantum physics. These were introduced in
\cite{FoulisB94}; an overview of the theory can be found in
\cite{DvurecenskijP00}.

Effect algebras form an abstract generalization of the unit interval $[0,1]
\subseteq \mathbb{R}$. This interval carries a partial addition: the sum of
two elements may or may not lie in the unit interval again. Furthermore, it
has a minimal and a maximal element, and complements with respect to the
maximal element. We capture the algebraic structure of $[0,1]$ in the notion
of an effect algebra.

\begin{definition}
    An \emph{effect algebra} consists of a set $A$ equipped with a partial
    binary operation $\eaplus$, a unary operation $(-)^\bot$ and elements
    $0,1 \in A$, such that:
    \begin{itemize}
        \item Commutativity: if $a \eaplus b$ is defined, then so is $b
            \eaplus a$, and $a \eaplus b = b \eaplus a$.
        \item Associativity: if $a \eaplus b$ and $(a \eaplus b) \eaplus c$
            are defined, then so are $b \eaplus c$ and $a \eaplus (b
            \eaplus c)$, and $(a \eaplus b) \eaplus c = a \eaplus (b
            \eaplus c)$.
        \item Zero: $0\eaplus a$ is always defined and equals $a$.
        \item Orthocomplement: for each $a \in A$, $a^\bot$ is the unique
            element for which $a \eaplus a^\bot = 1$.
        \item Zero-one law: if $a\eaplus 1$ is defined, then $a=0$.
    \end{itemize}
\end{definition}

Every effect algebra carries an order. Define $a \leq b$ if and only if there
exists an element $c$ such that $a \eaplus c$ exists and equals $b$. The
axioms for an effect algebra guarantee that $\leq$ is a partial order.

It is easy to see that the unit interval $[0,1]$ forms an effect algebra,
where $\eaplus$ is partial addition and the orthocomplement is given by
$a^\bot = 1-a$. Furthermore, there are several examples coming from quantum
logic. Each orthomodular lattice is an effect algebra, in which $a \eaplus b$
is defined if and only if $a$ and $b$ are disjoint, that is $a \wedge b = 0$.
In that case, $a \eaplus b$ is defined to be $a \vee b$.
The orthocomplement is already included in the structure of an orthomodular
lattice.

Also each unital C*-algebra induces an effect algebra. An element $a$ in a
C*-algebra $A$ is called \emph{positive} if it can be written as $a = b^*b$
for some $b\in A$.
Define an order on the collection of self-adjoint elements by putting $a \leq
b$ if and only if $b-a$ is positive. Then the unit interval $[0,1]_A = \{ a
\in A \mid a \text{ is self-adjoint and } 0 \leq a \leq 1 \}$ is an effect
algebra, where $a\eaplus b$ is defined if and only if $a + b \leq 1$, and in
that case $a\eaplus b = a+b$. The orthocomplement is $a^\bot = 1-a$.

The above example is a special case of a unit interval in an ordered
abelian group. An ordered abelian group is simply an abelian group equipped
with a partial order that is compatible with the addition. If $u$ is any
positive element in an ordered abelian group $A$, then $[0,u]$ forms an effect
algebra, again with partial addition as operation and $a^\bot = u-a$ as
orthocomplement. An effect algebra that is isomorphic to an interval in some
ordered abelian group is called an \emph{interval effect algebra}, see
\cite{BennettF97}.

There are several notions of morphisms between effect algebras; we will need
two of these. A function $f$ from $A$ to $B$ is simply called a
\emph{morphism} if:
\begin{itemize}
    \item $f$ preserves $0$, $1$, and complements.
    \item If $a \eaplus b$ is defined, then also $f(a) \eaplus f(b)$ is
        defined, and $f(a\eaplus b) = f(a) \eaplus f(b)$.
\end{itemize}
The notation $\EA$ stands for the category of effect algebras with morphisms
in this sense. A \emph{strong} morphism is a morphism for which the condition
that $f(a) \eaplus f(b)$ is defined implies that also $a \eaplus b$ is
defined. Sometimes ordinary morphisms are called \emph{weak} to distinguish
them from strong morphisms.  Most morphisms encountered in the theory of
effect algebras are weak, and the category of effect algebras with weak
morphisms has better properties than the category with strong morphisms, which
is why we usually omit the adjective ``weak''.

The distinction between these two kinds of morphisms can be made for all
partial algebraic structures. The terminology that we use here comes from
\cite{Gratzer68}. Sometimes strong morphisms are called closed morphisms,
following \cite{Burmeister86}. In the effect algebra literature, one sometimes
encounters the term monomorphism. However, we will avoid this term, due to
possible confusion with the categorical notion of monomorphism.

A \emph{subalgebra} of an effect algebra $B$ is a subset $A \subseteq B$
such that, whenever $a,a' \in A$ and $a \eaplus a'$ is defined in $B$,
then $a \eaplus a'$ lies in $A$. Equivalently, $A$ is a subalgebra of $B$
whenever the inclusion map $A \hookrightarrow B$ is a strong injective
morphism of effect algebras.

The category $\EA$ is complete and cocomplete, and possesses a well-behaved
tensor product, as proven in \cite{JacobsM12}. We will regularly need
products, coproducts, and tensor products, so we will describe these briefly
here. The product of effect algebras is simply the cartesian product with
pointwise operations.

To construct the coproduct of $A$ and $B$, put an equivalence relation $\sim$
on their disjoint union $A \coproduct B$ by identifying $0_A$ with $0_B$ and
$1_A$ with $1_B$. The coproduct $A + B$ is then the quotient $(A \coproduct B)/
\sim$. Denote the coprojections $A \to A+B$ and $B \to A+B$ by $\iota_A$ and
$\iota_B$, respectively. Then the sum of two elements $\iota_A(a_1)$ and
$\iota_A(a_2)$ is defined if and only if $a_1 \eaplus a_2$ is defined in $A$,
and in that case $\iota_A(a_1) \eaplus \iota_A(a_2) = \iota_A(a_1 \eaplus
a_2)$. Likewise one defines the sum of $\iota_B(b_1)$ and $\iota_B(b_2)$. The
sum of $\iota_A(a)$ and $\iota_B(b)$ is never defined for $a \neq 0,1$ and $b
\neq 0,1$. The orthocomplement in $A+B$ is derived from the ones in $A$ and
$B$.

A \emph{bimorphism} of effect algebras is a map $f: A \times B \to C$ that 
preserves addition in both variables separately, and satisfies $f(1,1) = 1$.
In \cite{JacobsM12} it is shown that any two effect algebras $A$ and $B$ have
a tensor product $A \otimes B$, which is constructed in such a way that
bimorphisms $A \times B \to C$ correspond bijectively to morphisms $A \otimes
B \to C$.

In \cite{FoulisGreechieBennett94} it has been shown that interval effect
algebras are stable under all the constructions described above. More
precisely we have the following.

\begin{proposition}
    \mbox{}
    \begin{enumerate}
        \item Any subalgebra of an interval effect algebra is again an
            interval effect algebra.
        \item If $A$ and $B$ are interval effect algebras, then so are $A
            \times B$, $A + B$, and $A \otimes B$.
    \end{enumerate}
    \label{prop:ConstructionsInterval}
\end{proposition}

In many physical theories, the duality between states and measurements plays
an important role. If we model measurements using an effect algebra $A$, then
a \emph{state} on $A$ is a (weak) morphism from $A$ to $[0,1]$. In many
examples of effect algebras, this indeed yields a reasonable notion of state
or probability measure. For example, a state on the effect algebra $[0,1]^X$,
where $X$ is a finite set, is a probability distribution on $X$. If $A$ is a
C*-algebra, then a state on the effect algebra $[0,1]_A$ is the same as a
state on $A$ in the C*-algebraic sense.

The state space of an effect algebra is especially interesting if it contains
enough information to recover the order on the effect algebra. This is
unfortunately not the case for all effect algebras. However, a result by
Dvure\v{c}enskij \cite{Dvurecenskij10} based on earlier work by Goodearl
\cite{Goodearl86} shows that it does hold for Archimedean interval effect
algebras. Roughly speaking, an effect algebra is Archimedean if it has no
infinitesimal elements. The notion can only be defined in the case of interval
effect algebras.

\begin{definition}
    An ordered abelian group is called \emph{Archimedean} if it satisfies the
    following property: whenever $nx \leq y$ for all $n\in \mathbb{N}$, then
    $x \leq 0$. An interval effect algebra is called Archimedean if its
    ambient group is Archimedean.
\end{definition}

The result by Dvure\v{c}enskij and Goodearl is as follows.

\begin{theorem}
    \label{ThmArchimedeanOrderDetermining}
    An interval effect algebra is Archimedean if and only if its state
    space is \emph{order-determining}, which means: if for all states $\sigma$
    we have $\sigma(a) \leq \sigma(b)$, then $a\leq b$.
\end{theorem}

\section{Tests}
\label{sec:Tests}

Boolean algebras are among the simplest examples of effect algebras. Certain
effect algebras can be obtained by gluing several Boolean algebras together.
These algebras become easier to analyze if we understand their constituent
Boolean algebras, and the gluing construction, so we will take a look at this
construction here.

We will frequently use the notion of a test on an effect algebra. An
$n$-\emph{test} on $A$ consists of $n$ elements $a_1, \ldots, a_n$ such that
$a_1 \eaplus \cdots \eaplus a_n$ is defined and equals $1$.
We introduce the following notation for tests:
\[ 
    \tests_n(A) = \{ (a_0, \ldots, a_n) \mid a_0 \eaplus \cdots \eaplus a_n
    = 1 \} 
\]
Note that $\tests_n(A)$ contains the $(n+1)$-tests; this convention will turn
out to be beneficial when defining cohomology of effect algebras. If all but
one elements of a test are known, then the final one is fixed since
orthocomplements in an effect algebra are unique. Therefore $\tests_n(A)$ is
isomorphic to the set
\[
    \{ (a_1, \ldots, a_n) \mid a_1 \eaplus \cdots \eaplus a_n
    \text{ is defined} \}.
\]

An \emph{orthoalgebra} is an effect algebra in which $a \eaplus a$ is never
defined, unless $a = 0$. All information contained in a finite orthoalgebra
can be conveniently organized into a \emph{Greechie diagram}.  Our description
of Greechie diagrams follows \cite{Holzer07}. More background on the topic can
be found in \cite{Kalmbach83,SvozilTkadlec96}. We need a generalization of the
notion of a graph, called a hypergraph. A graph consists of points and a set
of two-element subsets of the points, representing the edges. A hypergraph
generalizes this by dropping the requirement that the subsets have two points.

\begin{definition}
    A \emph{hypergraph} comprises a set $P$ of points and a
    set $H \subseteq \Pow(P)$, elements of which are called
    \emph{hyperedges} or \emph{lines}, such that
    $\bigcup H = P$ and $\varnothing \notin H$.
\end{definition}

Each hypergraph can be represented pictorially. To do this, we simply
draw a point for each point in the hypergraph. A hyperedge is drawn as a
smooth curve connecting all points in the corresponding hyperedge. For
example, consider the following two diagrams.
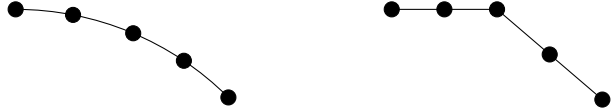
\begin{center}
    \begin{tikzpicture}
        \def\distance {5cm}

        \draw [fill=black] (90:4) circle (0.1);
        \draw [fill=black] (79:4) circle (0.1);
        \draw [fill=black] (67:4) circle (0.1);
        \draw [fill=black] (56:4) circle (0.1);
        \draw [fill=black] (45:4) circle (0.1);
        \draw (45:4) arc [start angle=45, end angle=90, radius=4];

        \pgftransformxshift{\distance}

        \draw [fill=black] (0,4) circle (0.1);
        \draw [fill=black] (0.7,4) circle (0.1);
        \draw [fill=black] (1.4,4) circle (0.1);
        \draw [fill=black] (2.1,3.4) circle (0.1);
        \draw [fill=black] (2.8,2.8) circle (0.1);
        \draw (0,4) -- (1.4,4) -- (2.8,2.8);
    \end{tikzpicture}
\end{center}
The diagram on the left represents a hypergraph with 5 points and one single
hyperedge containing all of those points. The diagram on the right represents
a hypergraph with 5 points and two hyperedges of 3 points, because it has a
corner.

To any orthoalgebra $A$ we can assign a hypergraph, called its Greechie
diagram.  A non-zero element $a$ of an effect algebra
is called an \emph{atom} if the only element lying below $a$ is
$0$.  A test on $A$ consisting of only atoms is called a \emph{maximal
    test}, since it has no refinements without zeroes. The
Greechie diagram of an orthoalgebra $A$ is a hypergraph with a point for each
atom, and a hyperedge for each maximal test.

\begin{example}
    \label{ex:firefly}
    The Greechie diagram
    
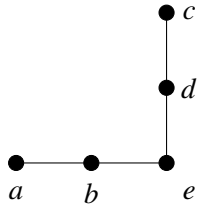
\begin{center}
        \begin{tikzpicture}
            \draw [fill=black] (0,0) circle (0.1);
            \draw [fill=black] (-1,0) circle (0.1);
            \draw [fill=black] (-2,0) circle (0.1);
            \draw [fill=black] (0,1) circle (0.1);
            \draw [fill=black] (0,2) circle (0.1);

            \node at (-2,-0.4) {$a$};
            \node at (-1,-0.4) {$b$};
            \node at (0.3,-0.4) {$e$};
            \node at (0.3,2) {$c$};
            \node at (0.3,1) {$d$};

            \draw (-2,0) -- (0,0);
            \draw (0,2) -- (0,0);
        \end{tikzpicture}
    \end{center}
    represents an orthoalgebra with 5 atoms $a$, $b$, $c$, $d$, $e$, in such a way that
    $\{ a,b,e \}$ and $\{ c,d,e \}$ are maximal tests. This means that $ a
    \eaplus b \eaplus e = 1$, $c \eaplus d \eaplus e = 1$, and that the sum of an
    atom in $\{ a,b \}$ and an atom in $\{ c,d \}$ is undefined. The condition $a
    \eaplus b \eaplus e = c \eaplus d \eaplus e$ implies that $a \eaplus b =
    c \eaplus d$. Thus the orthoalgebra consists of 12 elements $0$, $a$, $b$, $c$, $d,
    e$, $a \eaplus b = c \eaplus d$, $a \eaplus e$, $b \eaplus e$, $c \eaplus e$, $d
    \eaplus e$, $1$, with partial addition determined by the maximal tests.

    Note that the Greechie diagram is more concise than a description of
    the full orthoalgebra. This is the reason why finite orthoalgebras are
    often defined in terms of their Greechie diagrams.
\end{example}

The construction of an orthoalgebra from a Greechie diagram is made more
precise using the framework of test spaces, see for example
\cite{FoulisR72, FoulisB93}. It can also be interpreted as pasting
Boolean subalgebras
together, as discussed in \cite{HamhalterNP95, Navara00}.
Each maximal test in an orthoalgebra generates a maximal Boolean subalgebra,
elements of which are sums of its atoms. A maximal Boolean subalgebra is
called a \emph{block}. Conversely, the atoms of any block form a
maximal test.  Thus there is a one-to-one correspondence between blocks and
maximal tests in any finite orthoalgebra. Since each finite orthoalgebra is
completely determined by its atoms and maximal tests, it is the union of its
blocks.

When gluing blocks in an orthoalgebra together, it is often desirable to know
how tests on a union relate to tests on the constituents. The following result
gives such a relation.

\begin{proposition}
    Let $A$ and $B$ be subalgebras of an effect algebra $E$, such that $E = A
    \cup B$. Any test on $E$ is a test on $A$ or a test on $B$.
    \label{prop:TestUnion}
\end{proposition}
\begin{proof}
    Suppose that $(t_0, \ldots, t_n)$ is a test on $E$. Assume towards a
    contradiction that it is neither a test on $A$, nor a test on $B$.  Then
    there are $i$ and $j$ such that $t_i \notin A$ and $t_j \notin B$.  Since
    $(t_0, \ldots, t_n)$ is a test on the union, we have $t_i \in B \setminus
    A$ and $t_j \in A \setminus B$, and $t_i \eaplus t_j$ is defined in $E = A
    \cup B$.  Without loss of generality, assume that $t_i \eaplus t_j \in A$.
    Then $t_i \eaplus t_j \eaplus a = 1$ for some $a \in A$, so $t_i$ is the
    orthocomplement of $t_j \eaplus a$. The sum $t_j \eaplus a$ is defined in
    $E$, and both $t_j$ and $a$ lie in $A$. Since $A$ is a subalgebra of $E$,
    the sum $t_j \eaplus a$ is also defined in $A$.  Therefore $t_i = (t_j
    \eaplus a)^\bot$ also lies in $A$, which is a contradiction.
\end{proof}

We will now take a look at the structure of the collection of all tests on an
effect algebra. In particular, we will study the connection between tests and
cyclic sets.

Effect algebras are formally similar to abstract circles, introduced in
\cite{Moerdijk96}, which are an abstract generalization of the unit circle
$\circlegroup$. The following definition is easily seen to be equivalent to
the one given in \cite{Moerdijk96}.

\begin{definition}
    \label{DefAbstractCircle}
    An \emph{abstract circle} consists of a non-empty set $P$ of points, and
    for each two points $x,y$ a set $\Hom(x,y)$ of segments from $x$ to $y$.
    Furthermore, there are partial functions $\cup : \Hom(x,y) \times
    \Hom(y,z) \rightharpoonup \Hom(x,z)$, functions $(-)^\bot : \Hom(x,y) \to
    \Hom(y,x)$, and segments $0_x, 1_x \in \Hom(x,x)$ for each $x$. These are
    subject to the following requirements:
    \begin{itemize}
        \item Associativity: if $a \cup b$ and $(a \cup b) \cup c$ are
            defined, then so are $b \cup c$ and $a \cup (b \cup c)$, and $(a
            \cup b) \cup c = a \cup (b \cup c)$.
        \item Zero: for each $a\in \Hom(x,y)$, $0_x \cup a = a = a \cup 0_y$.
        \item Orthocomplement: for all $a \in \Hom(x,y)$ and $b \in \Hom(y,x)$,
            we have
            \[ a \cup b = 1_x \Longleftrightarrow a = b^\bot
            \Longleftrightarrow b = a^\bot \]
        \item Zero-one law: for any $a \in \Hom(x,y)$, if $a \cup 1_y$ is
            defined, then $a = 0_y$. Also, if $1_x \cup a$ is defined, then $a
            = 0_x$.
        \item Totality: for all $a \in \Hom(x,y)$ and $b \in \Hom(y,z)$, at
            least one of $a \cup b$ and $b^\bot \cup a^\bot$ exists.
        \item Trivial automorphisms: $\Hom(x,x) = \{0_x, 1_x\}$.
    \end{itemize}

    A morphism $F$ from an abstract circle $P$ to $Q$ consists of a function
    $P \to Q$ and functions $\Hom(x,y) \to \Hom(F(x),F(y))$, such that $F$
    preserves $0_x$, $1_x$, and the complement, and is subject to the
    following functoriality condition: whenever $a \cup b$ is defined, then
    also $F(a) \cup F(b)$ is defined, and $F(a \cup b) = F(a) \cup F(b)$.
\end{definition}

An important example is any subset of the unit circle $\circlegroup$.
For $x\neq y$, the set $\Hom(x,y)$ is a singleton, whose element represents
the circle segment from $x$ to $y$, counterclockwise. The homset $\Hom(x,x)$
has two elements $0_x$ and $1_x$, where $0_x$ represents the segment
consisting of the single point $x$, and $1_x$ represents a full circle.  The
composition of the segment from $x$ to $y$ and the segment from $y$ to $z$ is
given by gluing the segments, which is defined whenever the segments together
do not exceed the circle.

There is a common generalization of effect algebras and abstract circles
called \emph{effect algebroids}. An effect algebroid is defined in the same
way as an abstract circle, but without the conditions on Totality and Trivial
Automorphisms. Effect algebras and abstract circles are both extreme cases of
effect algebroids. An effect algebra is an effect algebroid with one point,
where the binary operation is additionally commutative. An abstract circle is
an effect algebroid with multiple points, but only one segment between any
two different points, and only two segments from a point to itself. Effect
algebroids give much insight in the connections between effect algebras
and abstract circles, but we will mainly work with ordinary effect algebras.
For a description of the theory from the more general viewpoint of effect
algebroids, see \cite{Roumen16}.

Abstract circles are closely connected to cyclic sets. To define these, first
let $\Cycl$ be a skeleton of the category of \emph{finite} abstract circles.
The unique abstract circle with $n$ points will be denoted by $\cycle{n}$.
The category $\Cycl$ is called the \emph{category of cycles}. A cyclic set is
defined to be a presheaf on $\Cycl$. The value of a presheaf $X : \Cycl\op \to
\Sets$ on the abstract circle $\cycle{n}$ will be written as $X_{n-1}$. This
is to ensure that indexing of the cyclic set starts at zero, because abstract
circles are required to be non-empty.

The definition of a cyclic set is often given in a more combinatorial way,
using generators and relations. We will sketch the description here, see e.g.\ 
\cite{Jones87} for more details. The category $\Cycl$ is generated by
morphisms of the following forms:
\begin{itemize}
    \item Face maps $\delta_i : \cycle{n-1} \to \cycle{n}$ for $i=0,\ldots,
        n-1$, where $\delta_i$ is the injection that skips the
        $i^{\mathrm{th}}$ point in $\cycle{n}$.
    \item Degeneracy maps $\sigma_i : \cycle{n+1} \to \cycle{n}$ for
        $i=0,\ldots, n-1$, where $\sigma_i$ is the surjection that hits the
        $i^{\mathrm{th}}$ point twice.
    \item Cyclic permutations $\tau : \cycle{n} \to \cycle{n}$ that map each
        point to the next point counterclockwise.
\end{itemize}
These morphisms are subject to certain relations. A cyclic set is then a
sequence of sets $(X_n)_{n \in \mathbb{N}}$ equipped with maps $d_i : X_n \to
X_{n-1}$ for $i=0,\ldots, n$, maps $s_i : X_{n-1} \to X_n$ for $i=0, \ldots,
n$, and maps $t: X_n \to X_n$ satisfying the duals of these relations.

The Yoneda embedding is a morphism from the category of finite abstract
circles to the category of cyclic sets. If $P$ is a finite abstract circle,
then $\Yoneda(P)_n$ consists of sequences of $n+1$ segments in $P$ that sum to
1. Because of the similarity between effect algebras and abstract circles, we
can also define an embedding $\EA \to \cSets$. In an effect algebra, a
sequence of $n$ elements that sums to 1 is called an $n$-test, as discussed 
above. Hence the effect algebraic analogue of $\Yoneda(P)_n$ is the set of
$(n+1)$-tests $\tests_n(A)$.

The family of sets $(\tests_n(A))_{n \in \mathbb{N}}$ forms a cyclic set.
The face maps of $\tests(A)$ are given by adding adjacent elements in a test: 
\begin{IEEEeqnarray*}{r.rCl+l}
    d_i : & \tests_n(A) & \to & \tests_{n-1}(A) & \\
          & (a_0, \ldots, a_n) & \mapsto &
            (a_0, \ldots, a_i \eaplus a_{i+1}, \ldots, a_n)
          & \text{for } 0\leq i \leq n-1 \\
          & (a_0, \ldots, a_n) & \mapsto &
            (a_n \eaplus a_0, \ldots, a_{n-1})
          & \text{for } i=n
\end{IEEEeqnarray*}
Degeneracies are given by inserting zeroes:
\begin{IEEEeqnarray*}{r.rCl+l}
    s_i : & \tests_{n-1}(A) & \to & \tests_n(A) & \\
          & (a_0, \ldots, a_{n-1}) & \mapsto &
            (a_0, \ldots, a_{i-1}, 0, a_i, \ldots, a_{n-1}) 
          & \text{for } 0 \leq i \leq n-1 \\
          & (a_0, \ldots, a_{n-1}) & \mapsto &
            (a_0, \ldots, a_{n-1}, 0)
          & \text{for } i=n
\end{IEEEeqnarray*}
Finally, cyclic permutations are defined by
\[ \lambda : \tests_n(A) \to \tests_n(A), \quad (a_0, \ldots, a_n) \mapsto
    (a_n, a_0, \ldots, a_{n-1}) \]

\section{Cyclic cohomology of an effect algebra}
\label{sec:CyclicCohomologyEffectAlgebra}

Effect algebras embed in cyclic sets,
and cyclic sets admit a natural cohomology theory called \emph{cyclic cohomology}.
Therefore it is reasonable to use cyclic cohomology also for effect algebras.
Cyclic cohomology was introduced by Connes in \cite{Connes83,Connes85}, see
also \cite{Loday84}. The book \cite{Loday98} contains an overview of the
theory.

The cohomology groups arising from a cyclic set are defined from a cochain
complex associated to the cyclic set. We will now describe this construction
for the cyclic set of tests $\tests(A)$. We will take coefficients in the
field $\mathbb{R}$, since some of our results only hold over this field of
coefficients. There are two versions of the definition of cyclic cohomology:
Connes' version from \cite{Connes83} is simpler, but only valid over fields
containing the rational numbers. Tsygan's version from \cite{Tsygan83} uses a
double complex and is more complicated, but also more general. Since we will
only be concerned with coefficients in $\mathbb{R}$, we will work with Connes'
definition.

Let $C^{\bullet}(A)$ be the complex
\[ \mathbb{R}^{\tests_0(A)} \overset{\delta^0}{\longrightarrow}
    \mathbb{R}^{\tests_1(A)} \overset{\delta^1}{\longrightarrow} \cdots \]
Elements of $\mathbb{R}^{\tests_n(A)}$ are functions from the $(n+1)$-tests to
$\mathbb{R}$ and are called $n$-\emph{cocycles}.
Observe that every effect algebra has exactly one 1-test, so
$\mathbb{R}^{\tests_0(A)}$ can be identified with $\mathbb{R}$.
Also, in a 2-test, each entry determines the other one via complementation, so
$\mathbb{R}^{\tests_1(A)}$ can be identified with $\mathbb{R}^{A}$. The
boundary maps are given by an alternating sum over the face maps.
Explicitly,
\[
    \delta^n(\alpha)(a_0, a_1, \ldots, a_n, a_{n+1}) =
    \sum_{i=0}^{n} (-1)^i \alpha(a_0, \ldots, a_i \eaplus a_{i+1},
    \ldots, a_{n+1}) 
    + (-1)^{n+1} \alpha(a_{n+1} \eaplus a_0, a_1, \ldots, a_{n}).
\]

We wish to consider only cocycles that are invariant under the action of
$\lambda$ defined above. In other words, take a subcomplex of $C^{\bullet}(A)$
consisting of those cocycles $\alpha$ for which
\[ \alpha(a_0, a_1, \ldots, a_n) = (-1)^{n} \alpha(a_n, a_0, \ldots,
    a_{n-1}). \]
The boundary maps $\delta^n$ send invariant cocycles to invariant cocycles, so this
indeed gives a well-defined subcomplex, denoted $C_{\lambda}^{\bullet}(A)$. The
cyclic cohomology of the effect algebra $A$ is the cohomology of
$C_{\lambda}^{\bullet}(A)$, that is,
$\HC^n(A) = \ker(\delta^{n}) / \im(\delta^{n-1})$.

Sometimes we will also be interested in the cohomology of the complex
$C^{\bullet}(A)$ itself, i.e.\ without taking the subcomplex of invariant
cocycles. The cohomology of $C^{\bullet}(A)$ is called the \emph{Hochschild
    cohomology} of $A$ and denoted $\HH^n(A)$. We will see that the Hochschild
cohomology of an effect algebra is not as well-behaved as its cyclic
cohomology. However, there are useful relations between Hochschild
cohomology and cyclic cohomology, for instance Connes' exact sequence
connecting the two. Therefore computing Hochschild cohomology is sometimes a
practical intermediate step for computing cyclic cohomology. 

\begin{example}
    \label{ex:CohomologyTwoElementAlgebra}
    We will determine the cohomology groups of the effect algebra $L_1 = \{
    0,1 \}$ via a direct computation. The $n$-tests on $L_1$ have a 1 at
    exactly one position, and are zero at all other positions. If $\alpha \in
    C_{\lambda}^n(L_1)$, then $\alpha$ is determined by its value on the test
    $(1,0,\ldots,0)$ by invariance. Hence each $C_{\lambda}^n(L_1)$ is a
    one-dimensional vector space.

    If $n$ is even, then $(\delta^n \alpha)(a_0, \ldots, a_{n+1})$ is an
    alternating sum with $n + 2$ terms. By invariance, all terms in the sum
    are equal, so because the sum is alternating and has an even number of
    terms, it is zero. Hence we have $\delta^n = 0$ for even $n$, and similarly
    $\delta^n$ is non-zero for odd $n$. Thus
    \[ \ker(\delta^n) =  
        \begin{cases}
            \mathbb{R} & \text{if $n$ is even} \\
            0 & \text{if $n$ is odd}
        \end{cases}
    \]
    and
    \[ \im(\delta^n) = 
        \begin{cases}
            0 & \text{if $n$ is even} \\
            \mathbb{R} & \text{if $n$ is odd.}
        \end{cases}
    \]
    Therefore
    \[ \HC^n(L_1) = 
        \begin{cases}
            \mathbb{R} & \text{if $n=0$} \\
            0 & \text{if $n>0$.}
        \end{cases}
    \]
\end{example}

We will look at the cohomology groups of an effect algebra $A$ in low degrees.
For $n=0$, the definition reduces to $\HC^0(A) = \ker(\delta^0 : \mathbb{R}
\to \mathbb{R}^A)$, since there is only one 1-test, and 2-tests correspond to
elements of $A$.  The boundary map $\delta^0$ satisfies $\delta^0(\alpha)(a) =
\alpha(a \eaplus a^\bot) - \alpha(a^\bot \eaplus a) = 0 $, hence $\HC^0(A)$
is always the ground field $\mathbb{R}$.

We continue with the first cohomology group $\HC^1(A)$. 
We will first rewrite the definition of $\HC^1$. Since the boundary map
$\delta^0$ is zero, $\HC^1(A)$ reduces to the kernel of $\delta^1$. We identify
1-cocycles $\alpha$ with maps from $A$ to $\mathbb{R}$, by letting a 2-test
$(a,b)$ correspond to the element $b\in A$. Invariance under the cyclic
permutation map $\lambda$ then means that $\alpha(a^{\bot}) =
-\alpha(a)$, and $\alpha \in \ker(\delta^1)$ means that $\alpha(b) -
\alpha(a \eaplus b) + \alpha(a) = 0$ whenever $a\eaplus b$ is defined.
Therefore
\[ \HC^1(A) \cong \{ \alpha : A \to \mathbb{R} \mid \alpha(a \eaplus b) =
    \alpha(a) + \alpha(b),\ \alpha(a^{\bot})= -\alpha(a) \}. \]

Recall that the state space of $A$ is the convex space of morphisms $\sigma: A
\to [0,1]$. These satisfy $\sigma(a \eaplus b) = \sigma(a) + \sigma(b)$ and
$\sigma(a^\bot) = 1 - \sigma(a)$. With the definition of the first cohomology
group written as above, we see that the state space is similar to the first
cohomology group. We will make the connection more precise. The state space of
an effect algebra is always a compact convex space, and hence it embeds in a
vector space over $\mathbb{R}$. We would like to prove that $\HC^1(A)$ (with
coefficients in $\mathbb{R}$) is the \emph{smallest} vector space that
contains a copy of $\St(A)$. 
This means that there exists an affine injection $i : \St(A)
\hookrightarrow \HC^1(A)$, such that for all affine injections $j : \St(A)
\hookrightarrow V$ into a vector space there exists a unique affine
injection $\varphi : \HC^1(A) \hookrightarrow V$ that makes the triangle
\begin{center}
    \begin{tikzpicture}
        \matrix [commutative diagram] {
            \node (St) {$\St(A)$}; &
            \node (Hc) {$\HC^1(A)$}; \\
            & \node (V) {$V$}; \\
        };
        \path [right hook->] (St) edge node [mor,above] {i} (Hc)
            edge node [mor,below left] {j} (V);
        \path [dashed, right hook->] (Hc) edge node [mor,right] {\varphi} (V);
    \end{tikzpicture}
\end{center}
commute.  Note that $\varphi$ is a map between vector spaces, but it is
usually not linear. We can only obtain an affine map between the vector
spaces.

Unfortunately this result does not hold for all effect algebras, for
instance it fails for the effect algebra of projections on a Hilbert space.
However, the result holds for many classes of well-behaved effect algebras. We
will first present a general result that provides a sufficient condition on $A$
that makes the statement true. This sufficient condition is hard to prove in
practice, so after proving the general result we will mention a large
class of effect algebras that satisfy the condition.

\begin{definition}
    A map $\varphi$ from an effect algebra $A$ into $\mathbb{R}$ is
    \emph{additive} if $\varphi(a \eaplus b) = \varphi(a) + \varphi(b)$
    whenever $a \eaplus b$ is defined.
\end{definition}

\begin{theorem}
    \label{thm:FirstCohomologyGroupStateSpace}
    Let $A$ be an effect algebra whose state space is non-empty. Suppose that
    every additive map $\alpha : A \to \mathbb{R}$ can be written as a
    difference of two positive additive maps $\alpha = \alpha_1 -
    \alpha_2$, where $\alpha_1, \alpha_2 : A \to \mathbb{R}_{\geq 0}$. Then
    $\HC^1(A)$ is the smallest vector space that contains a copy of the state
    space $\St(A)$.
\end{theorem}
\begin{proof}
    Fix a state $\sigma_0$ and use this to define an embedding $i : \St(A) \to
    \HC^1(A)$ by $i(\sigma) = \sigma - \sigma_0$. Then $i(\sigma)$ is linear
    because $\sigma$ and $\sigma_0$ are, and $i(\sigma)$ satisfies
    \[ i(\sigma)(a^{\bot}) = \sigma(a^{\bot}) - \sigma_0(a^{\bot}) = 1 -
        \sigma(a) - (1 - \sigma_0(a)) = - \sigma(a) + \sigma_0(a) =
        -i(\sigma)(a). \]
    Thus $i$ maps states to cocycles in $\HC^1(A)$, and $i$ is clearly
    injective and affine.

    Let $j: \St(A) \to V$ be an arbitrary affine injection. To define a map
    $\varphi : \HC^1(A) \to V$, take any $\alpha \in \HC^1(A)$. Since $\alpha$
    is additive and $\sigma_0$ is a state, $\alpha + \sigma_0$ is also
    additive. Using the hypothesis, express $\alpha + \sigma_0$ as a
    difference $\alpha + \sigma_0 = \alpha_1 - \alpha_2$, where $\alpha_1$ and
    $\alpha_2$ are positive additive maps. 
    To define $\varphi(\alpha)$, we distinguish several cases.
    \begin{itemize}
        \item Suppose that $\alpha_1(1)$ and $\alpha_2(1)$ are both non-zero.
            Define $\sigma_i(a) = \frac{\alpha_i(a)}{\alpha_i(1)}$ for $i =
            1,2$, which is a state.
            Then define $\varphi$ via $\varphi(\alpha) = \alpha_1(1)
            j(\sigma_1) - \alpha_2(1) j(\sigma_2)$.
        \item If $\alpha_1(1) = 0$ and $\alpha_2(1)$ is non-zero, define
            $\sigma_2(a) = \frac{\alpha_2(a)}{\alpha_2(1)}$ and put
            $\varphi(\alpha) = -\alpha_2(1) j(\sigma_2)$.
        \item Similarly, if $\alpha_1(1) \neq 0 $ and $\alpha_2(1) = 0$,
            define $\sigma_1(a) = \frac{\alpha_1(a)}{\alpha_1(1)}$ and put
            $\varphi(\alpha) = \alpha_1(1) j(\sigma_1)$.
        \item Finally, if $\alpha_1(1) = \alpha_2(1) = 0$, then let
            $\varphi(\alpha) = 0$.
    \end{itemize}

    The decomposition of $\alpha + \sigma_0$ need not be unique, so we have to
    prove that $\varphi$ is well-defined by showing that it does not depend
    on the choice of decomposition.
    Suppose that $\alpha + \sigma_0 = \alpha_1 - \alpha_2 = \alpha'_1 -
    \alpha'_2$. We will assume that all of $\alpha_1(1)$, $\alpha_2(1)$,
    $\alpha'_1(1)$ and $\alpha'_2(1)$ are non-zero; the other cases are
    easier.
    We have to prove that $\alpha_1(1) j(\sigma_1) - \alpha_2(1) j(\sigma_2) =
    \alpha'_1(1) j(\sigma'_1) - \alpha'_2(1) j(\sigma'_2)$. For this we use
    that $j$ preserves convex combinations, and that linear combinations can
    be made convex by normalization:
    \[ \frac{\alpha_1(1)}{\alpha_1(1) + \alpha'_2(1)} j(\sigma_1) +
        \frac{\alpha'_2(1)}{\alpha_1(1) + \alpha'_2(1)} j(\sigma'_2)
        = j \left( \frac{\alpha_1(1)\sigma_1 + \alpha'_2(1)\sigma'_2}
            {\alpha_1(1) + \alpha'_2(1)} \right) \\
        = j \left( \frac{\alpha_1 + \alpha'_2}{\alpha_1(1) +
                \alpha'_2(1)} \right) \]
    Now using $\alpha_1 + \alpha'_2 = \alpha'_1 + \alpha_2$ and rewriting back
    shows that this equals
    \[ \frac{\alpha'_1(1)}{\alpha'_1(1) + \alpha_2(1)} j(\sigma'_1) +
        \frac{\alpha_2(1)}{\alpha'_1(1) + \alpha_2(1)} j(\sigma_2), \]
    so $\alpha_1(1) j(\sigma_1) - \alpha_2(1)j(\sigma_2) = \alpha'_1(1)
    j(\sigma'_1) - \alpha'_2(1) j(\sigma'_2)$.

    The next step is to show that $\varphi$ makes the triangle commute, i.e.\ 
    $\varphi(i(\sigma)) = j(\sigma)$ for all states $\sigma$. A decomposition
    of $i(\sigma) + \sigma_0$ is just $\sigma - 0$, since $\sigma$ is a state
    and hence positive. Then $\varphi(i(\sigma)) = \sigma(1) j(\sigma) =
    j(\sigma)$, as required.

    It is easy to see that $\varphi$ is affine. To show that it is injective,
    suppose that $\varphi(\alpha) = \varphi(\alpha')$. Then $\alpha_1(1)
    j(\sigma_1) - \alpha_2(1) j(\sigma_2) = \alpha'_1(1) j(\sigma'_1) -
    \alpha'_2(1) j(\sigma'_2)$. By using normalization and affinity of $j$, we
    obtain $j \left( \frac{\alpha_1 + \alpha'_2}{\alpha_1(1) + \alpha'_2(1)}
    \right) = j \left( \frac{\alpha'_1 + \alpha_2}{\alpha'_1(1) +
            \alpha_2(1)} \right)$, and since $j$ is injective this gives
    $\alpha_1 - \alpha_2 = \alpha'_1 - \alpha'_2$. This means $\alpha =
    \alpha'$, proving injectivity of $\varphi$.

    Finally we have to prove that $\varphi$ is the unique morphism with this
    property. Suppose that an affine map $\psi : \HC^1(A) \to V$ satisfies
    $\psi \circ i = j$. Take $\alpha \in \HC^1(A)$ and decompose $\alpha +
    \sigma_0$ as $\alpha_1 - \alpha_2$ where both $\alpha_i$ are positive. We
    assume that $\alpha_1(1)$ and $\alpha_2(1)$ are both non-zero; the other
    cases are similar. Define $\sigma_i = \frac{\alpha_i}{\alpha_i(1)}$ as
    before. We have to show that $\psi(\alpha) = \alpha_1(1) j(\sigma_1) -
    \alpha_2(1) j(\sigma_2)$. Since $\psi \circ i = j$, we have $\psi(\sigma -
    \sigma_0) = j(\sigma)$ for all states $\sigma$. Therefore we are done if
    we can establish that $\alpha_1(1) \psi(\sigma_1 - \sigma_0) =
    \psi(\alpha) + \alpha_2(1) \psi(\sigma_2 - \sigma_0)$. We will prove a
    normalized version of this equality, that is,
    \[ \frac{\alpha_1(1)}{1 + \alpha_2(1)} \psi(\sigma_1 - \sigma_0) =
    \frac{1}{1 + \alpha_2(1)} \psi(\alpha) + \frac{\alpha_2(1)}{1 + \alpha_2(1)}
    \psi(\sigma_2 - \sigma_0). \]

    To prove this, first note that
    \[ 1 + \alpha_2(1) = \alpha(1) + \sigma_0(1) + \alpha_2(1) = \alpha_1(1), \]
    where we used that $\sigma_0$ is a state and that $\alpha(1) = -\alpha(0)
    = 0$ because $\alpha$ lies in $\HC^1(A)$. Furthermore,
    \begin{IEEEeqnarray*}{rCl}
        \alpha + \alpha_2(1)(\sigma_2 - \sigma_0)
        &=& \alpha + \alpha_2 - \alpha_2(1)\sigma_0 \\
        &=& \alpha_1 - \sigma_0 - \alpha_2(1)\sigma_0 \\
        &=& \alpha_1 - (1 + \alpha_2(1)) \sigma_0 \\
        &=& \alpha_1 - \alpha_1(1) \sigma_0 \\
        &=& \alpha_1(1) (\sigma_1 - \sigma_0).
    \end{IEEEeqnarray*}
    Because $\psi$ preserves convex combinations, it follows that
    \begin{IEEEeqnarray*}{rCl}
        \frac{1}{1 + \alpha_2(1)} \psi(\alpha) + \frac{\alpha_2(1)}{1 + \alpha_2(1)}
        \psi(\sigma_2 - \sigma_0)
        &=& \psi \left( \frac{\alpha + \alpha_2(1) (\sigma_2 - \sigma_0)}{1 +
                \alpha_2(1)} \right) \\
        &=& \psi \left( \frac{ \alpha_1(1) (\sigma_1 - \sigma_0)}{\alpha_1(1)} \right) \\
        &=& \psi(\sigma_1 - \sigma_0) \\
        &=& \frac{\alpha_1}{1 + \alpha_2(1)} \psi(\sigma_1 - \sigma_0).
    \end{IEEEeqnarray*}
    This finishes the proof that $\varphi$ is unique.
\end{proof}

The next result shows that all finite Archimedean interval effect algebras
satisfy the assumption in the previous theorem. By
Theorem~\ref{ThmArchimedeanOrderDetermining}, the state space of any such algebra is
non-empty. Therefore, for all finite Archimedean interval effect algebras $A$,
the first cohomology group $\HC^1(A)$ is the smallest vector space surrounding
$\St(A)$. 

\begin{proposition}
    \label{prop:FiniteEAPositivelyGenerated}
    If $A$ is a finite Archimedean interval effect algebra, then every
    additive map $\alpha : A \to \mathbb{R}$ can be expressed as the
    difference of two positive additive maps.
\end{proposition}
\begin{proof}
    The following proof is inspired by an analogous result for complemented
    lattices in \cite{Dvurecenskij78}, but modified to be suitable for effect
    algebras.

    Since $A$ is finite, it can be presented by a finite number of generators
    and relations. Let $X$ be a finite set of generators. The state space of
    $A$ consists of maps $X \to [0,1]$ subject to the relations. Therefore the
    state space is a compact convex space generated by a finite number of
    points. Let $\sigma_1, \ldots, \sigma_n$ be generators for the state space
    and define a state $\beta$ by $\beta = \frac{1}{n} \sigma_1 + \cdots +
    \frac{1}{n} \sigma_n$.
    We will show that $\beta$ is a faithful state,
    which means that $\beta(a) \neq 0$ for all $a \neq 0$. Assume that $a \neq
    0 $ but $\beta(a) = 0$. Then $\sigma_i(a) = 0$ for all $i = 1, \ldots, n$.
    But since the state space is generated by the states $\sigma_i$, this
    implies that all states $\sigma$ map $a$ to zero. By
    Theorem~\ref{ThmArchimedeanOrderDetermining}, this is only possible if
    $a=0$, contradicting our assumption that $a \neq 0$.

    We will use the faithful state $\beta$ to prove the proposition. Let
    $\alpha : A \to \mathbb{R}$ be an additive map. We may assume that
    $\alpha(a) < 0$ for some $a \in A$, since otherwise the claim is proven
    immediately. Let  
    \[ K = \frac{-\min \{ \alpha(a) \mid a \in A \}}{\min \{ \beta(a) \mid a
            \neq 0 \}} \in \mathbb{R}. \]
    Both minimums exists since $A$ is finite. The denominator is strictly
    positive, because $\beta$ is a faithful state. Also the numerator is
    strictly positive, since there is an $a \in A$ for which $\alpha(a) < 0$.
    Hence $K > 0$.

    We wish to write $\alpha$ as the difference of two positive additive maps
    $\alpha = \alpha_1 - \alpha_2$. Take $\alpha_2(a) = K \beta(a)$, which is
    positive since $K$ and $\beta$ are positive, and additive since $\beta$ is
    additive. Then let $\alpha_1 = \alpha + \alpha_2$. Clearly $\alpha_1$ is
    additive and $\alpha = \alpha_1 - \alpha_2$, so it is left to check that
    $\alpha_1$ is positive. Take any $b \neq 0$ in $A$. Then $\min \{ \beta(a)
    \mid a \neq 0\} \leq \beta(b)$, and since $\min \{ \alpha(a) \mid a \in A
    \}$ is negative, it follows that 
    \[ K \geq \frac{-\min \{ \alpha(a) \mid a \in A \} }{\beta(b)}. \]
    Therefore $\alpha_1(b) = \alpha(b) + K \beta(b) \geq \alpha(b) -
    \min \{ \alpha(a) \mid a \in A \} \geq 0$, where the last inequality uses
    that $\alpha(a)$ is negative for some $a$. This proves that $\alpha_1$ is
    a positive map, hence $\alpha$ is the difference of two positive maps.
\end{proof}

\section{Relative cohomology}
\label{sec:RelativeCohomology}

We shall define relative cohomology of effect algebras and discuss some
applications.

Let $B$ be an effect algebra and $A \subseteq B$ a subalgebra. Each test on
$A$ is in particular a test on $B$, so the collection of $(n+1)$-tests
$T_n(A)$ on $A$ can be considered as a subset of $T_n(B)$. This gives a
surjection $p^n : \Hom(T_n(B), \mathbb{R}) \to \Hom(T_n(A), \mathbb{R})$ by
restriction:
\[ p^n(\alpha) = \alpha|_{T_n(A)} \]
Since the map $p^n$ is compatible with cyclic permutations, it restricts to a
surjection $C_{\lambda}^n(B) \to C_{\lambda}^n(A)$, also denoted $p^n$ or $p$.

The kernel of $p^n$ consists of all invariant cochains $T_n(B) \to \mathbb{R}$
that are zero on $A$-tests, but not necessarily on $B$-tests. It fits in a
short exact sequence
\[ 0 \longrightarrow \ker(p^n) \longrightarrow C_{\lambda}^n(B)
    \overset{p^n}{\longrightarrow} C_{\lambda}^n(A) \longrightarrow 0. \]
The coboundary maps of the cochain complex $C_{\lambda}^{\bullet}(B)$ restrict to
$\ker(p^n)$, so the above is in fact a short exact sequence of cochain
complexes. The \emph{relative cohomology} of the pair $(B,A)$ is defined to be
the cohomology of $\ker(p^{\bullet})$. By general results from homological
algebra (see e.g.\ \cite{Weibel94}), the short exact sequence above gives rise
to a long exact sequence in cohomology:
\[ \cdots \longrightarrow \HC^{n-1}(A) \longrightarrow \HC^n(B,A)
    \longrightarrow \HC^n(B) \longrightarrow \HC^n(A) \longrightarrow
    \HC^{n+1}(B,A) \longrightarrow \cdots \]

As a first application of relative cohomology, we will show that trivial tests
can be ignored when calculating the cohomology of an effect algebra. A trivial
test is a test $(a_0, \ldots, a_n)$ in which exactly one $a_i$ is one and all
others are zero. To make the statement precise, consider the effect algebra
$L_1 = \{ 0,1 \}$. This can be embedded in any effect algebra $A$, since all
effect algebras have a zero and a one. The relative cohomology of the pair
$(A, L_1)$ is the cohomology of $\ker(p^{n})$, where $p^n : C_{\lambda}^n(A)
\to C_{\lambda}^n(L_1)$ is the restriction map. Since the tests on $L_1$ are
exactly the trivial tests, the kernel of $p^n$ consists of those cocycles that
are zero on trivial tests. Hence the claim that trivial tests can be ignored
in the calculation of cohomology groups amounts to the following.

\begin{proposition}
    \label{prop:CohomologyTrivialTests}
    For any effect algebra $A$ and any $n>0$, $\HC^n(A, L_1) \cong \HC^n(A)$.
\end{proposition}
\begin{proof}
    Look at the long exact sequence for the pair $(A, L_1)$. We have
    seen that the cohomology of $L_1$ is $\mathbb{R}$ in degree 0 and zero
    elsewhere. Hence around degree 1 the long exact sequence looks like:
    \[ \HC^0(A) \cong \mathbb{R} \overset{\alpha}{\longrightarrow} \HC^0(L_1)
        \cong \mathbb{R} \overset{\beta}{\longrightarrow} \HC^1(A, L_1)
        \overset{\gamma}{\longrightarrow} \HC^1(A) \longrightarrow 0 \]
    The group $\HC^0(A)$ consists of all cocycles that map the trivial 1-test
    $(1)$ to a real number, and the same holds for the group $\HC^0(L_1)$.
    Since $\alpha$ is a restriction map, it is the identity on $\mathbb{R}$
    here. From exactness at $\HC^0(L_1)$ it follows that $\beta = 0$. This in
    turn implies that $\ker(\gamma) = \im(\beta) = 0$, so $\gamma$ is
    injective. But $\gamma$ is also surjective since the sequence is exact at
    $\HC^1(A)$, so $\HC^1(A, L_1) \cong \HC^1(A)$. This proves the result for
    $n=1$.

    For an arbitrary $n>1$, consider the fragment of the long exact sequence
    around degree $n$:
    \[ \HC^{n-1}(L_1) \longrightarrow \HC^n(A, L_1) \longrightarrow \HC^n(A)
        \longrightarrow \HC^n(L_1) \]
    Since $\HC^{n-1}(L_1) = \HC^n(L_1) = 0$, we conclude that $\HC^n(A, L_1)
    \cong \HC^n(A)$.
\end{proof}

The above proposition is useful to show that cyclic cohomology preserves
coproducts of effect algebras. For this property, it is essential that we use
cyclic cohomology. For Hochschild cohomology the analogous result is false.

\begin{corollary}
    For any $n>0$, $\HC^n(A+B) = \HC^n(A) \oplus \HC^n(B)$.
\end{corollary}
\begin{proof}
    We will show that $\HC^n( A+B,L_1) \cong \HC^n(A,L_1) \oplus
    \HC^n(B,L_1)$; the result will then follow from the previous proposition.
    Call the cochain complex that computes $\HC^n(A,L_1)$
    $D^{\bullet}(A)$. Similarly there are cochain complexes
    $D^{\bullet}(B)$ and $D^{\bullet}(A+B)$. These
    complexes consist of all invariant cocycles that map trivial tests to zero.

    A test on a coproduct $A + B$ is either a trivial test, or a non-trivial
    test on $A$, or a non-trivial test on $B$. (Beware that we do not have
    $T_n(A+B) \cong T_n(A) + T_n(B)$, since $T_n(A+B)$ contains $n$ trivial
    tests, while the coproduct on the right-hand side contains $2n$ trivial
    tests.) Therefore $D^n(A+B) \cong D^n(A) \oplus
    D^n(B)$, from which the desired follows. 
\end{proof}

\section{K\"unneth sequence}
\label{sec:Kunneth}

To compute the cohomology groups of a product of two effect algebras, the
K\"unneth sequence is helpful. As before, we only consider cohomology with
coefficients in $\mathbb{R}$.

\begin{theorem}
    \label{thm:Kunneth}
    Let $A$ and $B$ be effect algebras. There is a long exact sequence
    \begin{IEEEeqnarray*}{rCl}
        \cdots & \longrightarrow &
        \HC^{n-1}(A \times B)
        \longrightarrow
        \bigoplus_{p+q = n-2} \HC^p(A) \otimes \HC^q(B) \\
        & \longrightarrow &
        \bigoplus_{p+q = n} \HC^p(A) \otimes \HC^q(B)
        \longrightarrow
        \HC^n(A \times B)
        \longrightarrow \cdots
    \end{IEEEeqnarray*}
\end{theorem}
\begin{proof}
    Tests on a product algebra satisfy $T_n(A \times B) \cong T_n(A) \times
    T_n(B)$. Therefore $\Hom(T_n(A \times B), \mathbb{R}) \cong \Hom(T_n(A),
    \mathbb{R}) \otimes \Hom(T_n(B), \mathbb{R})$, so to compute the
    cohomology of the product, we have to look at the cohomology of a tensor
    product of cyclic modules. According to the dual of \cite[Thm. 4.3.11]{Loday98},
    this can be computed using the sequence in the theorem.
\end{proof}

\begin{example}
    \label{ex:CohomologyPowerSetTwo}
    We will compute the cohomology of the power set effect algebra $\Pow(2) =
    L_1 \times L_1$. We have already seen in
    Example~\ref{ex:CohomologyTwoElementAlgebra} that the cohomology of $L_1$
    is $\mathbb{R}$ in degree zero, and vanishes elsewhere. The fragment of
    the K\"unneth sequence around degree 1 looks like:
    \begin{IEEEeqnarray*}{rCl}
        \IEEEeqnarraymulticol{3}{l}{
            (\HC^0(L_1) \otimes \HC^1(L_1)) \oplus
            (\HC^1(L_1) \otimes \HC^0(L_1)) } \\
        \quad &\longrightarrow&
        \HC^1(\Pow(2))
        \longrightarrow
        \HC^0(L_1) \otimes \HC^0(L_1) \\
        &\longrightarrow&
        (\HC^0(L_1) \otimes \HC^2(L_1)) \oplus
        (\HC^1(L_1) \otimes \HC^1(L_1)) \oplus
        (\HC^2(L_1) \otimes \HC^0(L_1))
    \end{IEEEeqnarray*}
    The outer groups in this sequence are zero, by the computation of the
    cohomology of $L_1$. It follows that $\HC^1(\Pow(2)) \cong \HC^0(L_1)
    \otimes \HC^0(L_1) \cong \mathbb{R}$. Furthermore, the cohomology in
    degree zero is $\mathbb{R}$ since this holds for all effect algebras, and
    from the K\"unneth sequence it can be deduced that it is zero in degrees
    at least two.
\end{example}

If the connecting morphisms in the above sequence are unknown, then applying
the theorem can be problematic. In this case, it may be easier to compute
cyclic cohomology using Hochschild cohomology as an intermediate step. In the
remainder of this section, we will use this technique to compute the cyclic
cohomology of a power set effect algebra $\Pow(m)$, which is a product of $m$
copies of $L_1$. First we observe that the K\"unneth formula for Hochschild
cohomology assumes a particularly simple form.

\begin{proposition}
    \label{prop:KunnethHochschild}
    Let $A$ and $B$ be effect algebras. Then
    \[ \HH^n(A \times B) \cong \bigoplus_{p+q = n} \HH^p(A) \otimes \HH^q(B). \]
\end{proposition}
\begin{proof}
    This follows from e.g.\ \cite[Thm. 3.6.3]{Weibel94}, using that we take
    coefficients in a field.
\end{proof}

We will also need a connection between cyclic and Hochschild cohomology, in
the case where we work with a product of copies of $L_1$. 

\begin{lemma}
    \label{lem:HochschildCyclicProduct}
    For any effect algebra $A$, $\HH^n(A) \cong \HC^n(A \times L_1)$.
\end{lemma}
\begin{proof}
    We will show that the complex computing $\HH(A)$ is isomorphic to the
    complex computing $\HC(A \times L_1)$. Define a map $f :
    \mathbb{R}^{\tests_n(A)} \to C_{\lambda}^n(A \times L_1)$ in the following way.
    Take an arbitrary $\alpha : \tests_n(A) \to \mathbb{R}$ and an arbitrary
    test $((a_0, k_0), \ldots, (a_n, k_n))$ on $A \times L_1$. A test on $L_1$
    has a 1 at exactly one position, and zeroes everywhere else. Let $i$ be the
    unique index for which $k_i = 1$. Then put
    \[ (f\alpha)((a_0, k_0), \ldots, (a_n, k_n)) = (-1)^{in} \alpha(a_i, 
        a_{i+1}, \ldots, a_n, a_0, \ldots, a_{i-1}). \]

    To show that $f\alpha$ actually lies in $C_{\lambda}^n(A \times L_1)$, we
    have to prove that it is invariant under cyclic permutations, i.e.
    \[ (f\alpha)((a_n,k_n), (a_0,k_0), \ldots, (a_{n-1}, k_{n-1})) =
    (-1)^n (f\alpha)((a_0,k_0), \ldots, (a_n,k_n)). \]
    Suppose that the $i^\mathrm{th}$ entry of the test $(k_0, \ldots, k_n)$
    satisfies $k_i = 1$, and $i < n$. Then the $(i+1)^\mathrm{th}$ entry of
    $(k_n, k_0, \ldots, k_{n-1})$ has value $1$. Hence
    \begin{IEEEeqnarray*}{rCl}
        (f\alpha)((a_n,k_n), (a_0,k_0), \ldots, (a_{n-1}, k_{n-1})) 
        &=&
        (-1)^{(i+1)n} \alpha(a_i, \ldots, a_n, a_0, \ldots, a_{i-1}) \\
        &=& (-1)^n (-1)^{in} \alpha(a_i, \ldots, a_n, a_0, \ldots, a_{i-1}) \\
        &=& (-1)^n (f\alpha)((a_0,k_0), \ldots, (a_n,k_n))
    \end{IEEEeqnarray*}
    A similar computation shows that the result still holds if $i=n$.

    Now we will verify that $f$ is a chain map from the Hochschild complex to
    the cyclic complex. To achieve this, we have to check that
    \[ (f\delta\alpha)((a_0,k_0), \ldots, (a_n,k_n)) =
    (\delta f \alpha)((a_0,k_0), \ldots, (a_n,k_n)). \]
    First assume that $k_0 = 1$. Then the left-hand side of this equation
    becomes
    \[ 
        (\delta\alpha)(a_0, \ldots, a_n)
        = \sum_{j=0}^{n-1} (-1)^j \alpha(a_0, \ldots, a_j \eaplus a_{j+1},
        \ldots, a_{n})
        + (-1)^{n} \alpha(a_{n} \eaplus a_0, a_1, \ldots, a_{n-1}).
    \]
    The right-hand side equals
    \begin{IEEEeqnarray*}{l}
        \sum_{j=0}^{n-1} (-1)^j (f\alpha)((a_0,k_0), \ldots,
        (a_j \eaplus a_{j+1}, k_j \eaplus k_{j+1}),
        \ldots, (a_{n},k_n)) \\
        \quad \negmedspace {}
        + (-1)^{n} (f\alpha)((a_{n} \eaplus a_0, k_n \eaplus k_0),
        (a_1,k_1), \ldots, (a_{n-1},k_{n-1})).
    \end{IEEEeqnarray*}
    In each term of this sum, the first entry of the test has $1$ as second
    component. Therefore it is equal to the left-hand side.

    Now assume that $k_i = 1$ for some $i \neq 0$. We can reduce this to the
    previous case by permuting the tests cyclically. Since $f(\delta\alpha)$ is
    invariant under cyclic permutations, we have 
    \[  
        (f\delta\alpha)((a_0,k_0), \ldots, (a_n,k_n))
        = (-1)^{in} (f\delta\alpha)((a_i,k_i), \ldots, (a_n,k_n),
        (a_0,k_0), \ldots, (a_{n-1},k_{n-1})).
    \]
    Furthermore, since $f\alpha$ is invariant and $\delta$ maps invariant
    cochains to invariant cochains, also $\delta f \alpha$ is invariant under
    cyclic permutations. Hence
    \[  
        (\delta f \alpha)((a_0,k_0), \ldots, (a_n,k_n))
        = (-1)^{in} (\delta f \alpha)((a_i,k_i), \ldots, (a_n,k_n),
        (a_0,k_0), \ldots, (a_{n-1},k_{n-1})).
    \]
    But in the test $((a_i,k_i), \ldots, (a_n, k_n),(a_0,k_0), \ldots,
    (a_{n-1}, k_{n-1}))$, the first entry has a $1$ as second component, so we
    are back in the previous case. This shows that $f$ is a chain map.

    The final step is proving that $f$ is a bijection. For injectivity,
    suppose that $f\alpha = f\beta$. Then for each test $((a_0,k_0), \ldots,
    (a_n,k_n))$ we have
    \[ \alpha(a_i, \ldots, a_n, a_0, \ldots, a_{i-1}) =
    \beta(a_i, \ldots, a_n, a_0, \ldots, a_{i-1}). \]
    Let $(a_0, \ldots, a_n)$ be
    an arbitrary test on $A$. Take the test $(k_0, \ldots, k_n)$ on $L_1$
    defined by $k_0 = 1$ and $k_i = 0$ for $i \neq 0$. This yields
    $\alpha(a_0, \ldots, a_n) = \beta(a_0, \ldots, a_n)$. Since $(a_0, \ldots,
    a_n)$ was arbitrary, $f$ is injective.

    To establish surjectivity, let $\beta : \tests_n(A \times L_1) \to
    \mathbb{R}$ be a map invariant under cyclic permutations.
    Define $\alpha : \tests_n(A) \to \mathbb{R}$ by
    \[ \alpha(a_0, \ldots,
    a_n) = \beta((a_0,1), (a_1, 0), \ldots, (a_n,0)). \]
    In order to show that $f\alpha = \beta$, take a test $((a_0,k_0), \ldots,
    (a_n,k_n))$ with $k_i = 1$. Then
    \begin{IEEEeqnarray*}{rCl}
        (f\alpha)((a_0,k_0) \ldots, (a_n,k_n))
        &=& (-1)^{in} \alpha(a_i, \ldots, a_n, a_0, \ldots, a_{i-1}) \\
        &=& (-1)^{in} \beta((a_i,1), (a_{i+1},0), \ldots, (a_n,0),
        (a_0,0), \ldots, (a_{i-1},0)) \\
        &=& \beta((a_0,k_0), \ldots, (a_n,k_n))
    \end{IEEEeqnarray*}
    where we used invariance of $\beta$ in the last step.
\end{proof}

\begin{example}
    \label{ex:CohomologyPowerSet}
    We will compute the cyclic cohomology groups of all power set effect
    algebras $\Pow(m)$. First we will determine their Hochschild cohomology.
    For $m=1$, apply Lemma~\ref{lem:HochschildCyclicProduct} and
    Example~\ref{ex:CohomologyPowerSetTwo} to find $\HH^n(\Pow(1)) \cong
    \HC^n(\Pow(2))$, which is $\mathbb{R}$ in degrees 0 and 1, and zero in all
    higher degrees. By applying the K\"unneth formula from
    Proposition~\ref{prop:KunnethHochschild} with induction to $m$, we obtain
    $\HH^n(\Pow(m)) \cong \mathbb{R}^{\binom{m}{n}}$. From
    Lemma~\ref{lem:HochschildCyclicProduct} it now follows that
    $\HC^n(\Pow(m)) \cong \HH^n(\Pow(m-1)) \cong \mathbb{R}^{\binom{m-1}{n}}$.
\end{example}

\section{Mayer--Vietoris sequence}
\label{sec:MayerVietoris}

A finite orthoalgebra is the union of its maximal Boolean subalgebras, as
discussed in Section~\ref{sec:Tests}. Since these are generated by the
maximal tests, the orthoalgebra is completely determined by its atoms and
maximal tests. In this section we will establish a Mayer--Vietoris sequence for
the cyclic cohomology of an effect algebra, which relates the cohomology of a
union to the cohomology of the constituents and their intersection. Since we
already know the cohomology of finite Boolean algebras, this yields a
technique for computing the cohomology of any finite orthoalgebra. Using the
Mayer--Vietoris sequence is usually a very efficient way to determine the
cohomology groups, since it only involves the atoms and the maximal tests,
instead of the collection of all tests on the effect algebra.

\begin{theorem}
    \label{thm:MayerVietoris}
    Let $A$ and $B$ be subalgebras of an effect algebra $E$, such that $E = A
    \cup B$.  Then there is a long exact sequence
    \[
        \cdots
        \longrightarrow
        \HC^{n-1}(A \cap B)
        \longrightarrow
        \HC^n(E)
        \longrightarrow
        \HC^n(A) \oplus \HC^n(B)
        \longrightarrow
        \HC^n(A \cap B)
        \longrightarrow
        \HC^{n+1}(E)
        \longrightarrow
        \cdots
    \]
\end{theorem}
\begin{proof}
    We shall construct a short exact sequence
    \[ 0 \longrightarrow C_{\lambda}^n(E)
        \overset{\varphi}{\longrightarrow}
        C_{\lambda}^n(A) \oplus C_{\lambda}^n(B)
        \overset{\psi}{\longrightarrow} C_{\lambda}^n(A \cap B)
        \longrightarrow 0, \]
    which will induce the desired long exact sequence in cohomology.
    Define $\varphi : C_{\lambda}^n(E) \to C_{\lambda}^n(A) \oplus
    C_{\lambda}^n(B)$ by restricting to tests on the subalgebras, i.e.\
    $\varphi(\alpha) = \left( \alpha |_{\tests_n(A)}, \alpha |_{\tests_n(B)}
    \right)$. The map $\psi$ is defined by $\psi(\alpha,\beta) = \alpha
    |_{\tests_n(A \cap B)} - \beta |_{\tests_n(A \cap B)}$.

    Now we will show that the maps $\varphi$ and $\psi$ yield a short exact
    sequence. To show that $\varphi$ is injective, suppose that
    $\varphi(\alpha) = \varphi(\beta)$. Then $\alpha(t) = \beta(t)$ for all
    tests $t$ on $A$, and all tests $t$ on $B$. Hence, by
    Proposition~\ref{prop:TestUnion}, $\alpha = \beta$, establishing
    injectivity.

    We continue by proving surjectivity of $\psi$. Take any $\gamma \in
    C_{\lambda}^n(A \cap B)$. Define $\alpha \in C_{\lambda}^n(A)$ and $\beta
    \in C_{\lambda}^n(B)$ as follows: for any test $t$ on $A \cap B$, define
    $\alpha(t) = \frac{1}{2} \gamma(t)$ and $\beta(t) = - \frac{1}{2}
    \gamma(t)$. On all tests that do not lie completely inside $A \cap B$,
    $\alpha$ and $\beta$ are zero. Then, for each test $t$ on $A \cap B$,
    $\psi(\alpha,\beta)(t) = \alpha(t) - \beta(t) = \frac{1}{2} \gamma(t) +
    \frac{1}{2} \gamma(t) = \gamma(t)$, so $\psi$ is surjective.

    Finally we will show that the sequence is exact in the middle. If $\alpha
    \in C_{\lambda}^n(E)$, then $\alpha |_{\tests_n(A)}$ and $\alpha
    |_{\tests_n(B)}$ agree on the intersection $\tests_n(A \cap B)$. It follows
    that $(\psi \circ \varphi)(\alpha) = 0$, hence $\im(\varphi) \subseteq
    \ker(\psi)$. Conversely, suppose that $\alpha \in C_{\lambda}^n(A)$ and
    $\beta \in C_{\lambda}^n(B)$ agree on $\tests_n(A \cap B)$. We have to show
    that both are restrictions of some $\gamma \in C_{\lambda}^n(E)$. Let $t$
    be a test on $E$. By Proposition~\ref{prop:TestUnion}, $t$ is either a
    test on $A$ or a test on $B$. If it is a test on $A$, define $\gamma(t) =
    \alpha(t)$; if it is a test on $B$, define $\gamma(t) = \beta(t)$. Then
    $\gamma$ is well-defined because $\alpha$ and $\beta$ agree on the
    intersection, and it restricts to $\alpha$ and $\beta$ on $\tests_n(A)$ and
    $\tests_n(B)$, respectively. This concludes the proof that $\im(\varphi) =
    \ker(\psi)$.
\end{proof}

\begin{example}
    \label{ex:CohomologyFirefly}
    We will compute the cohomology groups of the effect algebra from
    Example~\ref{ex:firefly}.
    Call the effect algebra $E$, let
    $A$ be the subalgebra generated by the atoms $a,b,e$, and let $B$ be the
    subalgebra generated by $c,d,e$. Then $E = A \cup B$, and $A \cong B \cong
    \Pow(3)$. Furthermore, $A \cap B$ consists of the four elements $0$, $e$,
    $a \eaplus b = c \eaplus d$, and $a \eaplus b \eaplus e = c \eaplus d
    \eaplus e = 1$, so it is isomorphic to $\Pow(2)$.  Plugging this
    information into the Mayer--Vietoris sequence gives
    \begin{IEEEeqnarray*}{l}
        \HC^0(\Pow(2))
        \overset{\partial_0}{\longrightarrow}
        \HC^1(E)
        \overset{\alpha}{\longrightarrow}
        \HC^1(\Pow(3)) \oplus \HC^1(\Pow(3))
        \overset{\beta}{\longrightarrow}
        \HC^1(\Pow(2)) \\
        \quad \overset{\partial_1}{\longrightarrow}
        \HC^2(E)
        \overset{\gamma}{\longrightarrow}
        \HC^2(\Pow(3)) \oplus \HC^2(\Pow(3))
        \overset{\delta}{\longrightarrow}
        \HC^2(\Pow(2))
    \end{IEEEeqnarray*}
    Recall from Example~\ref{ex:CohomologyPowerSet} that $\HC^n(\Pow(m)) \cong
    \mathbb{R}^{\binom{m-1}{n}}$.

    Since the coboundary map $\delta^0$ is always zero, the connecting
    homomorphism $\partial_0$ is zero as well. From exactness of the
    Mayer--Vietoris sequence it follows that $\HC^1(E) \cong \im(\alpha) =
    \ker(\beta)$. The first cohomology group of an effect algebra consists of
    additive maps into $\mathbb{R}$ that map $1$ to $0$. Since every additive
    map $\Pow(2) \to \mathbb{R}$ can be extended to an additive map $\Pow(3)
    \to \mathbb{R}$, $\beta$ is surjective, hence $\HC^1(E) \cong
    \mathbb{R}^3$.

    Similarly we can compute the second cohomology group. Surjectivity of
    $\beta$ gives $\partial_1 = 0$. Furthermore $\HC^2(\Pow(2)) = 0$, hence
    $\HC^2(E) \cong \HC^2(\Pow(3)) \oplus \HC^2(\Pow(3)) \cong
    \mathbb{R}^2$. Since all higher cohomology groups of $\Pow(3)$ are zero,
    all groups $\HC^n(E)$ for $n \geq 3$ are zero as well.
\end{example}

The Mayer--Vietoris sequence can be applied repeatedly to find the cohomology
of orthoalgebras with more than two blocks. However, one has to be careful
that all unions of blocks encountered at intermediate stages are actual
subalgebras, since otherwise Theorem~\ref{thm:MayerVietoris} does not apply.
We give an example where this phenomenon plays a role.

\begin{example}
    \label{ex:bike}
    Consider the orthoalgebra $E$ with Greechie diagram
    
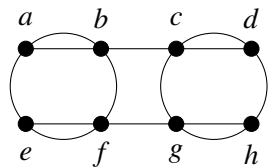
\begin{center}
        \begin{tikzpicture}
            \draw [fill=black] (0,0) circle (0.1);
            \draw [fill=black] (1,0) circle (0.1);
            \draw [fill=black] (2,0) circle (0.1);
            \draw [fill=black] (3,0) circle (0.1);
            \draw [fill=black] (0,1) circle (0.1);
            \draw [fill=black] (1,1) circle (0.1);
            \draw [fill=black] (2,1) circle (0.1);
            \draw [fill=black] (3,1) circle (0.1);

            \node at (0,1.4) {$a$};
            \node at (1,1.4) {$b$};
            \node at (2,1.4) {$c$};
            \node at (3,1.4) {$d$};
            \node at (0,-0.4) {$e$};
            \node at (1,-0.4) {$f$};
            \node at (2,-0.4) {$g$};
            \node at (3,-0.4) {$h$};

            \draw (0,0) -- (3,0);
            \draw (0,1) -- (3,1);
            \draw (0.5,0.5) circle (0.707);
            \draw (2.5,0.5) circle (0.707);
        \end{tikzpicture}
    \end{center}
    Naively, one could try to compute the cohomology of $E$ by adding one
    block at the time, for instance by first using Mayer--Vietoris to obtain
    the cohomology of the left diagram, and then using the result to obtain
    the cohomology of the right diagram:
    
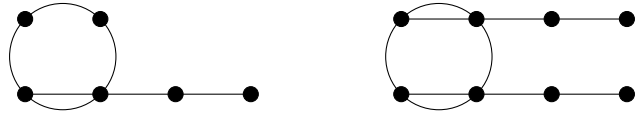
\begin{center}
        \begin{tikzpicture}
            \def\distance {5cm}

            \draw [fill=black] (0,0) circle (0.1);
            \draw [fill=black] (1,0) circle (0.1);
            \draw [fill=black] (2,0) circle (0.1);
            \draw [fill=black] (3,0) circle (0.1);
            \draw [fill=black] (0,1) circle (0.1);
            \draw [fill=black] (1,1) circle (0.1);

            \draw (0,0) -- (3,0);
            \draw (0.5,0.5) circle (0.707);

            \pgftransformxshift{\distance}

            \draw [fill=black] (0,0) circle (0.1);
            \draw [fill=black] (1,0) circle (0.1);
            \draw [fill=black] (2,0) circle (0.1);
            \draw [fill=black] (3,0) circle (0.1);
            \draw [fill=black] (0,1) circle (0.1);
            \draw [fill=black] (1,1) circle (0.1);
            \draw [fill=black] (2,1) circle (0.1);
            \draw [fill=black] (3,1) circle (0.1);

            \draw (0,0) -- (3,0);
            \draw (0,1) -- (3,1);
            \draw (0.5,0.5) circle (0.707);
        \end{tikzpicture}
    \end{center}
    Finally, use the cohomology of the right diagram to obtain the cohomology
    of $E$. However, this fails because the diagram on the right is not a
    subalgebra of $E$. Consider the atoms labeled $c$ and $g$ in $E$. Their
    sum is defined in $E$, since both lie on the right circle. But $c \eaplus
    g$ is not defined in the diagram on the right, since there is no hyperedge
    containing both $c$ and $g$. Therefore this diagram does not represent a
    subalgebra of $E$, and the Mayer--Vietoris sequence cannot be applied.

    To solve this problem, we have to build up $E$ in a different way.
    Consider the following subalgebras of $E$:
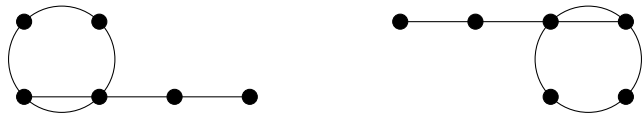
    \begin{center}
        \begin{tikzpicture}
            \def\distance {5cm}

            \draw [fill=black] (0,0) circle (0.1);
            \draw [fill=black] (1,0) circle (0.1);
            \draw [fill=black] (2,0) circle (0.1);
            \draw [fill=black] (3,0) circle (0.1);
            \draw [fill=black] (0,1) circle (0.1);
            \draw [fill=black] (1,1) circle (0.1);

            \draw (0,0) -- (3,0);
            \draw (0.5,0.5) circle (0.707);

            \pgftransformxshift{\distance}

            \draw [fill=black] (2,0) circle (0.1);
            \draw [fill=black] (3,0) circle (0.1);
            \draw [fill=black] (0,1) circle (0.1);
            \draw [fill=black] (1,1) circle (0.1);
            \draw [fill=black] (2,1) circle (0.1);
            \draw [fill=black] (3,1) circle (0.1);

            \draw (0,1) -- (3,1);
            \draw (2.5,0.5) circle (0.707);
        \end{tikzpicture}
    \end{center}
    Call the one on the left $A$ and the one on the right $B$.
    Note that both $A$ and $B$ are actual subalgebras of $E$.
    The diagrams represent isomorphic algebras, and their cohomology can be
    computed in the same way as in Example~\ref{ex:firefly}, yielding:
    \begin{center}
        \begin{tabular}{r|ccccc}
            $n$ & 0 & 1 & 2 & 3 & $\geq 4$ \\
            \hline
            $\HC^n(A),\ \HC^n(B)$ & $\mathbb{R}$ & $\mathbb{R}^4$ &
            $\mathbb{R}^5$ & $\mathbb{R}^2$ & $0$
        \end{tabular}
    \end{center}
    Since $A$ and $B$ are subalgebras and $E = A \cup B$, the Mayer--Vietoris
    sequence applies. The intersection $A \cap B$ is generated under addition
    by the elements $a, b, (a \eaplus b)^\bot, g, h, (g \eaplus h)^\bot$.
    Since $(a \eaplus b)^\bot = c \eaplus d = (g \eaplus h)^\bot$, the
    intersection has 5 atoms, and its Greechie diagram is
    
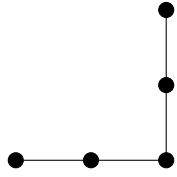
\begin{center}
        \begin{tikzpicture}
            \draw [fill=black] (0,0) circle (0.1);
            \draw [fill=black] (-1,0) circle (0.1);
            \draw [fill=black] (-2,0) circle (0.1);
            \draw [fill=black] (0,1) circle (0.1);
            \draw [fill=black] (0,2) circle (0.1);

            \draw (-2,0) -- (0,0);
            \draw (0,2) -- (0,0);
        \end{tikzpicture}
    \end{center}
    We determined the cohomology of this algebra in the previous example. From
    a Mayer--Vietoris argument it follows that $E$ has the following
    cohomology:
    \begin{center}
        \begin{tabular}{r|ccccc}
            $n$ & 0 & 1 & 2 & 3 & $\geq 4$ \\
            \hline
            $\HC^n(E)$ & $\mathbb{R}$ & $\mathbb{R}^5$ &
            $\mathbb{R}^8$ & $\mathbb{R}^4$ & $0$
        \end{tabular}
    \end{center}
\end{example}

\section{Generalized Mayer--Vietoris principle}
\label{sec:GeneralizedMayerVietoris}

Theorem~\ref{thm:MayerVietoris} only gives information about unions of two
subalgebras.  Applying the theorem repeatedly to get information about unions
of more than two subalgebras can be problematic, as witnessed by
Example~\ref{ex:bike}. The problem is that the union of two subalgebras need
not be a subalgebra again. Therefore it is sometimes desirable to have a
generalization of the above statement applicable to unions of an arbitrary
number of subalgebras. We will use an effect algebraic version of the
generalized Mayer--Vietoris principle from \cite{BottTu82}. It applies to
finite orthoalgebras, and gives a method to determine their cohomology from
the cohomology of their blocks.

Let $E$ be a finite orthoalgebra. Then $E$ can be written as a union $E = B_1
\cup \cdots \cup B_m$ of its blocks. We consider cocycles on the intersections
$B_{i_1} \cap \cdots \cap B_{i_k}$, for $1 \leq i_1 < \cdots < i_k \leq m$.
Our goal will be to prove that there is a long exact sequence
\[ 
    0 \longrightarrow
    C_{\lambda}^n(E) \longrightarrow
    \bigoplus_i C_{\lambda}^n(B_i) \longrightarrow
    \bigoplus_{i_1 < i_2} C_{\lambda}^n(B_{i_1} \cap B_{i_2})
    \longrightarrow \bigoplus_{i_1 < i_2 < i_3} C_{\lambda}^n(B_{i_1} \cap B_{i_2} \cap
    B_{i_3}) \longrightarrow \cdots
\]
This sequence generalizes the short exact sequence constructed in the proof of
the binary Mayer--Vietoris sequence by also including terms for intersections
of more than two subalgebras.

First we describe the maps involved in the sequence.
There is a restriction map $r : C_{\lambda}^n(E) \to
\bigoplus_i C_{\lambda}^n(B_i)$, whose $i^\mathrm{th}$ component maps $\alpha
\in C_{\lambda}^n(E)$ to $\alpha |_{\tests_n(B_i)}$. Furthermore, we define maps
\[ 
    \delta_k :
    \bigoplus_{i_1 < \cdots < i_k}
    C_{\lambda}^n(B_{i_1} \cap \cdots \cap B_{i_k})
    \to
    \bigoplus_{i_1 < \cdots < i_{k+1}}
    C_{\lambda}^n(B_{i_1} \cap \cdots \cap B_{i_{k+1}}) \]
for $k=1,2,\ldots$.
To define $\delta_k$ on a sequence 
$\overbar{\alpha} = \left( \alpha_{i_1 \ldots i_k} \right)_{i_1 < \cdots < i_k} $,
let the component of $\delta_k(\overbar{\alpha})$ with index $i_1 < \cdots <
i_{k+1}$ be 
\[ 
    \sum_{j=1}^{k+1} (-1)^{j+1}
        \left. \alpha_{i_1 \ldots \widehat{i_j} \ldots i_{k+1}}
        \right|_{\tests_n(B_{i_1} \cap \cdots \cap B_{i_{k+1}})}
\]
Here the hat $\widehat{i_j}$ means that the index $i_j$ has been omitted.

It is helpful to work out what this map does in low degrees. Firstly, the map
\[ \delta_1 : \bigoplus_i C_{\lambda}^n(B_i) \to \bigoplus_{i<j}
    C_{\lambda}^n(B_i \cap B_j) \]
takes as input a sequence $(\alpha_i)$ of maps $\tests_n(B_i) \to \mathbb{R}$,
for $i=1,\ldots,m$. The output is a sequence $(\beta_{ij})$ for $i<j$, where
$\beta_{ij} : \tests_n(B_i \cap B_j) \to \mathbb{R}$ is the map $\alpha_j -
\alpha_i$ restricted to tests on the intersection $B_i \cap B_j$. Secondly,
the map
\[ \delta_2 : \bigoplus_{i<j} C_{\lambda}^n(B_i \cap B_j) \to \bigoplus_{i<j<k}
    C_{\lambda}^n(B_i \cap B_j \cap B_k) \]
maps a sequence $(\alpha_{ij})$, indexed by $i<j$, to the sequence
$(\beta_{ijk})$, indexed by $i<j<k$, where $\beta_{ijk}$ is the restriction of
$\alpha_{jk} - \alpha_{ik} + \alpha_{ij}$.

\begin{proposition}
    [Generalized Mayer--Vietoris Principle]
    Let $E$ be a finite orthoalgebra with blocks $B_1, \ldots, B_m$. Then the
    sequence
    \[  
        0 \longrightarrow
        C_{\lambda}^n(E)
        \overset{r}{\longrightarrow}
        \bigoplus_i C_{\lambda}^n(B_i) 
        \overset{\delta_1}{\longrightarrow}
        \bigoplus_{i_1 < i_2} C_{\lambda}^n(B_{i_1} \cap B_{i_2})
        \overset{\delta_2}{\longrightarrow}
        \bigoplus_{i_1 < i_2 < i_3} C_{\lambda}^n(B_{i_1} \cap B_{i_2} \cap
        B_{i_3})
        \overset{\delta_3}{\longrightarrow}
        \cdots
    \]
    is exact.
    \label{prop:GeneralizedMayerVietoris}
\end{proposition}
\begin{proof}
    To prove that $r$ is injective, suppose that $r(\alpha) = r(\beta)$ for
    certain $\alpha,\beta \in C_{\lambda}^n(E)$. Then, for each $i=1,\ldots,
    m$ and each test $s$ on $B_i$, we have $\alpha(s) = \beta(s)$. We have to
    show that $\alpha$ and $\beta$ are the same on all tests on $E$. But if
    $t$ is a test on $E$, then its entries generate a Boolean subalgebra of
    $E$. By a standard application of Zorn's Lemma, this subalgebra can be
    enlarged to a block, which has to be one of the blocks $B_i$. Thus $t$ is
    a test on $B_i$, and hence $\alpha(t) = \beta(t)$.

    The next step is proving exactness at $\bigoplus_i C_{\lambda}^n(B_i)$.
    Using the explicit description of $\delta_1$ preceding the lemma, we see
    that
    \[ (\delta_1(r(\alpha)))_{i<j} = r(\alpha)_j - r(\alpha)_i |_{\tests_n(B_i
            \cap B_j)}. \]
    The maps $r(\alpha)_i$ and $r(\alpha)_j$ agree on the intersection $B_i
    \cap B_j$, since they are both restrictions of the same map $\alpha$.
    Therefore $\delta_1 \circ r = 0$, or equivalently, $\im(r) \subseteq
    \ker(\delta_1)$.

    For the reverse inclusion, suppose that $\overbar{\alpha} \in
    \ker(\delta_1)$. Then $\alpha_i(t) = \alpha_j(t)$ for all tests $t$ on
    $B_i \cap B_j$. We seek an $\alpha \in C_{\lambda}^n(E)$ such that $\alpha
    |_{\tests_n(B_i)} = \alpha_i$ for all $i$. For a test $t$ on $E$, define
    $\alpha(t)$ as follows: since $t$ is a test on $E$, it is a test on some
    block $B_i$. Define $\alpha(t)$ to be $\alpha_i(t)$. The condition
    $\alpha_i(t) = \alpha_j(t)$ shows that this is independent of the choice
    of block, making $\alpha$ well-defined. It is clear that $\alpha$
    restricts to $\alpha_i$ on $B_i$, finishing the proof that $\im(r) =
    \ker(\delta_1)$.

    Now we will show that $\im(\delta_{k-1}) = \ker(\delta_k)$ for $k \geq 2$.
    From a standard computation it follows that $\delta_k \circ \delta_{k-1} =
    0$. Suppose that a sequence $\left( \alpha_{i_1 \ldots i_k} \right)_{i_1 <
    \cdots < i_k}$ lies in $\ker(\delta_k)$. That means that
    \begin{equation}
        \label{eq:MayerVietorisCocycle}
        \sum_{j=1}^{k+1} (-1)^{j+1} \alpha_{i_1 \ldots \widehat{i_j} \ldots i_{k+1}} = 0 
    \end{equation}
    on $\tests_n(B_{i_1} \cap \cdots \cap B_{i_k})$, for all $i_1 < \cdots <
    i_{k+1}$.

    First we extend the definition of $\alpha$ to not necessarily increasing
    sequences of indices by stipulating that interchanging two indices gives
    a minus sign:
    \[ \alpha_{i_1 \ldots i_j \ldots i_{j'} \ldots i_k} = -\alpha_{i_1 \ldots
        i_{j'} \ldots i_j \ldots i_k} \]
    In particular that means that a repeated index always gives zero.

    Define $\beta_{i_1 \ldots i_{k-1}}$ on $B_{i_1} \cap \cdots \cap
    B_{i_{k-1}}$ in the following way: given a test $t \in \tests_n(B_{i_1}
    \cap \cdots \cap B_{i_{k-1}})$, let $N(t) = \{ j \mid t \in \tests_n(B_j)
    \}$. Then define
    \[ \beta_{i_1 \ldots i_{k-1}}(t) = \frac{1}{\# N(t)} \sum_{j \in N(t)}
        \alpha_{j i_1 \ldots i_{k-1}}(t). \]
    Here we implicitly used the convention about not necessarily increasing
    sequences of indices.

    To check that $\delta_{k-1}(\beta) = \alpha$, observe that
    \[
        \left(\delta_{k-1}(\beta)(t)\right)_{i_1 < \cdots < i_k}
        = \sum_{j=1}^k (-1)^{j+1} \beta_{i_1 \ldots \widehat{i_j} \ldots i_k}(t)
        = \sum_{j=1}^k \sum_{\ell \in N(t)} \frac{(-1)^{j+1}}{\# N(t)}
        \alpha_{\ell i_1 \ldots \widehat{i_j} \ldots i_k}.
    \]
    Condition \eqref{eq:MayerVietorisCocycle} with indices $\ell, i_1,
    \ldots, i_k$ becomes
    \[  
        \alpha_{i_1 \ldots i_k} -
        \sum_{j=1}^k (-1)^{j+1} \alpha_{\ell i_1 \ldots \widehat{i_j} \ldots
            i_{k}} = 0.
    \]
    Consequently,
    \[ \left(\delta_{k-1}(\beta)(t)\right)_{i_1 < \cdots < i_k}
        = \frac{1}{\# N(t)} \sum_{\ell \in N(t)} \alpha_{i_1 \ldots i_k}
        = \alpha_{i_1 \ldots i_k} \qedhere \]
\end{proof}

In Example~\ref{ex:CohomologyFirefly}, the cohomology groups become zero above
a certain degree. This is reminiscent of topological cohomology theories,
where cohomology groups in degree higher than the dimension of a space are
zero.  There is a similar result for cohomology of effect algebras, where the
dimension is replaced by the height.

\begin{definition}
    The \emph{height} of an effect algebra $A$ is the highest
    $n$ for which there is a chain $0 = a_0 < a_1
    < \cdots < a_n = 1$ in $A$.  If such $n$ does not exist, we say that $A$
    has infinite height. The height of $A$ is denoted
    $h(A)$.
\end{definition}

If $A$ is a finite orthoalgebra, then it can be represented using its atoms
and maximal tests. The height of $A$ is then the length of the longest test,
since a maximal test $(a_0, \ldots, a_n)$ gives a chain
\[ 0 < a_0 < a_0 \eaplus a_1 < \ldots < a_0 \eaplus \cdots \eaplus a_n = 1.  \]

\begin{theorem}
    [Height Theorem]
    \label{thm:height}
    Let $E$ be a finite orthoalgebra. For any $n \geq h(E)$, the cohomology
    group $\HC^n(E)$ is zero.
\end{theorem}
\begin{proof}
    First note that the Height Theorem holds for finite Boolean algebras: any
    finite Boolean algebra is a power set $\Pow(m)$, and according to
    Example~\ref{ex:CohomologyPowerSet}, the Height Theorem holds for
    $\Pow(m)$.

    If $E$ is a finite orthoalgebra, then it can be written as a union of
    blocks $E = B_1 \cup \cdots \cup B_m$.
    Proposition~\ref{prop:GeneralizedMayerVietoris} gives a long exact
    sequence
    \[ 0 \longrightarrow C_{\lambda}^n(E) \overset{\delta_0}{\longrightarrow}
        A_1 \overset{\delta_1}{\longrightarrow} A_2
        \overset{\delta_2}{\longrightarrow} \cdots, \]
    where $A_k = \bigoplus_{i_1 < \cdots < i_k} C_{\lambda}^n(B_{i_1} \cap
    \cdots \cap B_{i_k})$, and $\delta_0 = r$. For each $k \geq 1$, this gives
    a short exact sequence
    \[ 0 \longrightarrow \im(\delta_{k-1}) \longrightarrow A_k
        \overset{\delta_k}{\longrightarrow} \im(\delta_k) \longrightarrow 0. \]
    This in turn gives for each $k$ a long exact sequence in cohomology:
    \[
        \cdots \longrightarrow \HC^{n-1}(\im \delta_k) \longrightarrow
        \HC^n(\im \delta_{k-1}) \longrightarrow
        \HC^n(A_k) 
        \longrightarrow \HC^n(\im \delta_k) \longrightarrow \HC^{n+1}(\im
        \delta_{k-1}) \longrightarrow
        \cdots
    \]

    Since $E$ is finite, there exists $k$ such that $A_{k'} = 0$ for all
    $k'>k$. We will show that $\HC^{n-k+j}(\im \delta_{k-j}) = 0$ for each
    $j=1,\ldots,k-1$, by induction to $j$. To prove the claim for $j=1$, first
    we will show that $\HC^{n-k+1}(A_k) = 0$.  Finite Boolean algebras are
    fixed by their height, so if $B$ and $B'$ are different Boolean
    subalgebras of $E$, then $h(B \cap B') \leq h(B)-1, h(B')-1$. Using this
    fact repeatedly yields
    \[ h(B_{i_1} \cap \cdots \cap B_{i_k}) \leq h(B_{i_1}) - k + 1 \leq h(E) -
    k + 1 \leq n - k + 1. \]
    Therefore, by the Height Theorem for finite Boolean algebras,
    $\HC^{n-k+1}(A_k)$ is zero. Now look at the following fragment of the long
    exact sequence obtained earlier:
    \[ \HC^{n-k}(\im \delta_k) \longrightarrow \HC^{n-k+1}(\im \delta_{k-1})
        \longrightarrow
        \HC^{n-k+1}(A_k) \]
    Since $A_{k+1} = 0$, the map $\delta_k$ must be the zero map, hence
    $\HC^{n-k}(\im \delta_k) = 0$. We just showed that $\HC^{n-k+1}(A_k)$ is
    zero as well. By exactness, the term in the middle must also be zero,
    proving the first step in the induction.

    Now suppose that $\HC^{n-k+j}(\im \delta_{k-j}) = 0$ for a certain $j$.
    Then, using a similar argument as in the base case, it can be shown that
    $\HC^{n-k+j+1}(A_{k-j})$ is zero. Look at the following fragment of the
    long exact sequence:
    \[ \HC^{n-k+j}(\im \delta_{k-j}) \longrightarrow \HC^{n-k+j+1}(\im \delta_{k-(j+1)})
        \longrightarrow \HC^{n-k+j+1}(A_{k-j}) \]
    The outer terms are zero, so the inner term is zero too, finishing the
    induction argument.

    We know that $\HC^{n-k+j}(\im \delta_{k-j}) = 0$ for each
    $j=1,\ldots,k-1$. In particular, taking $j=k-1$, we obtain $\HC^{n-1}(\im
    \delta_1) = 0$. There is a short exact sequence
    \[ 0 \longrightarrow C_{\lambda}^n(E) \longrightarrow A_1 \longrightarrow
        \im(\delta_1) \longrightarrow 0, \]
    hence a fragment of a long exact sequence
    \[ \HC^{n-1}(\im \delta_1) \longrightarrow \HC^n(E) \longrightarrow \HC^n(A_1) \]
    We already noted that the term on the left is zero. By the Height Theorem
    for Boolean algebras, the term on the right is zero, hence $\HC^n(E) = 0$,
    which is what we wanted to prove.
\end{proof}

\section{Applications}
\label{sec:Applications}

Many no-go theorems in physics can be phrased in terms of morphisms between
effect algebras. We will show how cohomology helps to study these no-go
theorems. 

To keep the setting concrete, we will focus on the Bell scenario. The
following description of the Bell experiment is based on
\cite{StatonU15}. In the setup there are two observers, Alice and Bob.
Alice can perform either of two measurements $a$ and $a'$, with possible
outcomes $0$ and $1$. The event ``Alice performs measurement $a$ and obtains
outcome $i$'' will be denoted by $a_i$, and similarly we define $a'_i$. Bob
can also perform either of two measurements $b$ and $b'$, again with possible
outcomes $0$ and $1$. The notations $b_i$ and $b'_i$ have the expected
meanings. After both Alice and Bob have chosen a measurement, there are four
possible joint outcomes: $(0,0)$, $(0,1)$, $(1,0)$, and $(1,1)$. Each of these
is obtained with a certain probability, indicated in the following table:
\begin{center}
    \begin{tabular}{c|cccc}
        & $(0,0)$ & $(0,1)$ & $(1,0)$ & $(1,1)$ \\
        \hline
        $(a,b)$   & 1/2 & 0 & 0 & 1/2 \\
        $(a,b')$  & 3/8 & 1/8 & 1/8 & 3/8 \\
        $(a',b)$  & 3/8 & 1/8 & 1/8 & 3/8 \\
        $(a',b')$ & 1/8 & 3/8 & 3/8 & 1/8
    \end{tabular}
\end{center}
This table of probabilities cannot be reproduced by classical physics, but
there is a quantum mechanical setup realizing exactly these probabilities.
This is the content of Bell's famous theorem showing that quantum mechanics is
fundamentally different from classical mechanics, see
\cite{Bell64,AbramskyB11}.

The effect algebraic description of the Bell experiment is as follows. All
events for Alice can be collected in an effect algebra $E_A$ with elements $0,
a_0, a_1, a'_0, a'_1, 1$. Since Alice always obtains outcome 0 or 1, the sums
$a_0 \eaplus a_1$ and $a'_0 \eaplus a'_1$ are defined and equal to 1. All
other non-trivial sums are undefined, since Alice cannot perform the
measurements $a$ and $a'$ at the same time.
Thus $E_A$ is isomorphic to the coproduct effect algebra $\Pow(2) + \Pow(2)$.
It can be shown that this is the free effect algebra on two elements.
Similarly we construct an effect algebra $E_B$ for Bob's measurements, with
elements $0, b_0, b_1, b'_0, b'_1, 1$.  Since Bob can perform essentially the
same measurements as Alice, $E_B$ is isomorphic to $E_A$.  The effect algebra
representing the full experiment is $E := E_A \otimes E_B$, since composite
systems are modeled by tensor products.

Bell's Theorem states that there is a probability distribution on this system
that cannot be reproduced by classical physics. The probability distribution
amounts to a state on $E$. More precisely, the above table of probabilities
gives rise to a state that maps e.g.\ $a_i \otimes b'_j$ to the probability
that Alice obtains outcome $i$ when she picks measurement $a$, and Bob obtains
outcome $j$ when he picks measurement $b'$.

The measurements on a classical physical system are
given by an effect algebra of the form $\Pow(X)$ for some set $X$. Thus Bell's
Theorem says that there exists a state $\sigma : E \to [0,1]$ that does not
factor through any $\Pow(X)$:
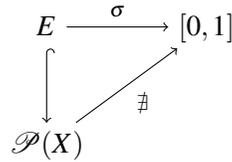
\begin{center}
    \begin{tikzpicture}
        \matrix [commutative diagram] {
            \node (E) {$E$}; &
            \node (interval) {$[0,1]$}; \\
            \node (PX) {$\Pow(X)$}; \\
        };
        \path [->] (E) edge node [mor, above] {\sigma} (interval);
        \path [right hook->] (E) edge (PX);
        \path [->] (PX) edge node [mor, below right] {\nexists} (interval);
    \end{tikzpicture}
\end{center}

In general, no-go theorems are about extending a state $\sigma : A \to [0,1]$
to a state on a larger effect algebra $B$, via an inclusion $i : A
\hookrightarrow B$. This inclusion map may be weak, i.e.\ it may not be an
actual inclusion of a subalgebra. We will now apply the cohomology theory of
effect algebras to study when extensions of states exist. Our approach is
similar to the one in \cite{AbramskyMS11}, but we use cyclic
cohomology of effect algebras instead of sheaf cohomology.

Let $A$ and $B$ be finite Archimedean interval effect algebras, and let $i : A
\hookrightarrow B$ be a weak injective morphism. Note that this assumption is
satisfied in the case of the Bell effect algebra: the power set $\Pow(2)$ is
clearly an interval effect algebra. Since the Bell effect algebra $E$ is
obtained from $\Pow(2)$ using coproducts and tensor products, it is an
interval effect algebra by Proposition~\ref{prop:ConstructionsInterval}, and
it is straightforward to check that $E$ is Archimedean.

Look at the following fragment of the long exact sequence of the pair $(B,A)$:
\[ \cdots \longrightarrow \HC^1(B) \longrightarrow \HC^1(A)
    \overset{\partial}{\longrightarrow} \HC^2(B,A) \longrightarrow \HC^2(B)
    \longrightarrow \cdots \]
By Theorem~\ref{thm:FirstCohomologyGroupStateSpace} and
Proposition~\ref{prop:FiniteEAPositivelyGenerated}, there exists an embedding
$j : \St(A) \to \HC^1(A)$, given by $j(\sigma) = \sigma - \sigma_0$ for some
fixed state $\sigma_0$. The map $j$ and the connecting homomorphism $\partial$
from the long exact sequence determine whether a state on $A$ extends to a
state on $B$.

\begin{theorem}
    Let $i : A \hookrightarrow B$ be a weak injective morphism between finite
    Archimedean interval effect algebras, and let $\sigma : A \to [0,1]$ be a
    state. If $\sigma$ extends to a state $\tau : B \to [0,1]$ for which $\tau
    \circ i = \sigma$, then the cohomology class $\partial(j(\sigma)) \in
    \HC^2(A,B)$ is zero.
\end{theorem}
\begin{proof}
    It is useful to have an explicit description of the connecting
    homomorphism $\partial$. Take a cohomology class $x \in \HC^1(A)$ and
    represent it by a map $\varphi : A \to \mathbb{R}$ satisfying
    $\varphi(a^{\bot}) = -\varphi(a)$. Since $i$ is injective, $\varphi$
    extends to a map $\psi : B \to \mathbb{R}$ with $\psi \circ i = \varphi$
    and $\psi(b^{\bot}) = -\psi(b)$. Applying the coboundary map $\delta$ to
    $\psi$ gives the 2-cocycle
    \[ (\delta\psi)(b,b') = \psi(b') - \psi(b \eaplus b') - \psi(b), \]
    which is defined on all pairs $(b,b')$ for which $b \eaplus b'$ exists.
    Then $\partial(x)$ is the relative cohomology class represented by
    $\delta\psi$.

    Suppose that the state $\sigma \in \St(A)$ extends to a state $\tau$ on
    $B$. Let $\tau_0$ be any state on $B$, and let $\sigma_0 = \tau_0 \circ
    i$. This gives the embedding $j(\sigma) = \sigma - \sigma_0$.
    Since $\tau$ extends $\sigma$, we have $(\tau-\tau_0) \circ i = \sigma -
    \sigma_0$, and $\tau-\tau_0$ is an additive map preserving complements.
    Therefore, by our description of the connecting homomorphism,
    $\partial(j(\sigma)) = \delta(\tau-\tau_0)$. But since $\tau-\tau_0$ is
    additive, its coboundary is zero, as required.
\end{proof}

Unfortunately, the converse does not hold, so false positives may arise.

\section{Order cohomology}
\label{sec:OrderCohomology}

Cyclic cohomology provides a necessary criterion for extending a state on an
effect algebra to a larger one, but not a sufficient criterion. The problem is
that positivity of the state is not encoded in the first cohomology group.
One can show that the coboundary of a state is zero if and only if it extends
to a signed state, i.e.\ one with possibly negative values. We will now define
a new cohomology theory of effect algebras that takes order, and hence
positivity, into account. This will lead to a necessary and sufficient
criterion for extending states.

The ideas behind cohomology of effect algebras that takes the order into
account have been studied before in \cite{Pulmannova06} and \cite{FeldmanW98},
although both of these only define a structure that behaves like a second
cohomology group. Our definition is a variation of Pulmannov\'a's cohomology
from \cite{Pulmannova06}, but generalized to give cohomology in arbitrary
degrees.

Defining cohomology of effect algebras with coefficients in an
ordered abelian group involves morphisms between these two structures.
Therefore we need a common generalization of effect algebras and ordered
abelian groups, to ensure that both live in the same category. Similar
structures have been considered in \cite{Pulmannova06,Wilce95}.

An \emph{ordered partial commutative monoid} is a partial commutative monoid
$A$ equipped with a positive cone $P \subseteq A$, for which:
\begin{itemize}
    \item $0 \in P$.
    \item If $a,b \in P$ and $a\eaplus b$ is defined, then $a \eaplus b \in
        P$.
    \item For $a,b \in P$, if $a \eaplus b = 0$, then $a = b = 0$.
\end{itemize}
We will write $A\positive$ for the positive cone $P$ of $A$. Any ordered
partial commutative monoid carries an order defined by $a \leq b$ if and only
if there exists $c \in A\positive$ such that $a \eaplus c = b$. It is
straightforward to show that this forms a partial order.

\begin{examples}
    \mbox{}
    \begin{enumerate}
        \item Any ordered abelian group is an ordered partial commutative
            monoid, in which the addition operation is total, and in which
            every element has an inverse.
        \item Any effect algebra $A$ is an ordered partial commutative monoid.
            The positive cone is simply all of $A$.
        \item Any partial commutative monoid $A$ can be made into an ordered
            partial commutative monoid by endowing it with the trivial cone
            $\{ 0 \}$. The resulting order is an antichain. The resulting
            structure is called a \emph{discrete} partial commutative monoid
            and denoted $\Disc(A)$.
    \end{enumerate}
\end{examples}

A morphism of ordered partial commutative monoids is just a morphism of their
underlying partial monoids. Such a morphism $f : A \to B$ is called
\emph{positive} if $f(A\positive) \subseteq B\positive$. A morphism is
positive if and only if it preserves the order. Furthermore, we say that $f$
is \emph{strong} if the condition that $f(a) \eaplus f(b)$ is defined implies
that also $a \eaplus b$ is defined.

\begin{definition}
    Let $f : A \to B$ be a morphism between ordered partial commutative
    monoids. The \emph{precone} of $f$ is $\precone(f) = f^{-1}(B\positive)
    \subseteq A$.
\end{definition}

The precone of a morphism $f : A \to B$ is again an ordered partial
commutative monoid, with addition and order inherited from $A$.
The restricted morphism $f |_{\precone(f)}$ is always a positive morphism, so
the precone construction is a way to transform non-positive morphisms into
positive morphisms, albeit in a somewhat trivial way.

If $B$ is discrete, then the precone of $f$ is simply its kernel. Hence
precones generalize kernels to the ordered setting. The kernel is a
fundamental operation for many constructions in homological algebra. We will
see that many results from homological algebra generalize to the setting of
ordered abelian groups, or ordered partial commutative monoids, by replacing
all kernels with precones.

The fundamental notion from homological algebra is a chain complex. Since we
will mainly use cohomology, we will work with cochain complexes. In the
ordered setting we define a 
\emph{cochain complex} to be a sequence
\[ 0 \longrightarrow A_0 \overset{\delta}{\longrightarrow}
    A_1 \overset{\delta}{\longrightarrow} \cdots, \]
where each $A_i$ is an ordered abelian group, each $\delta$ is a (not necessarily
positive) homomorphism, and $\delta \circ \delta = 0$. Define the collection of
$n$-cocycles by $\Zord{n}(A) = \{ a \in A_n \mid a \in \precone(\delta) \}$. The
index $\leq$ indicates that we take the order into account by using a precone
instead of a kernel. Since $\delta \circ \delta = 0$, we have $\im(\delta)
\subseteq \ker(\delta)
\subseteq \precone(\delta)$, so we can define order cohomology as
\[ \Hord{n}(A) = \precone(\delta) / \im(\delta). \]

The precone of a morphism between ordered abelian groups is an ordered
commutative monoid. The equivalence relation defined above is compatible with
addition, but not with the order,
so $\Hord{n}(A)$ is a commutative monoid.

In ordinary homological algebra, the cohomology of a quotient complex is
related to the cohomology of the larger complex via relative cohomology. We
will define relative cohomology of ordered abelian groups here, and show that
there is a sequence that captures some of its properties.

Let $p : B^{\bullet} \to A^{\bullet}$ be a surjective positive morphism of
cochain complexes. Then $p$ restricts to a map $\Zord{n}(B) \to \Zord{n}(A)$
because it is positive. Define the collection of relative cocycles by
$\Zord{n}(A,B) = \precone(\delta) \cap \precone(p)$. Put an equivalence
relation $\sim$ on $\Zord{n}(A,B)$ by $a \sim b$ if and only if there exists
$c$ such that $a-b = \delta(c)$ and $p(c) = 0$. Then the relative cohomology
of the pair $(B^{\bullet},A^{\bullet})$ is the quotient $\Hord{n} =
\Zord{n}/\sim$.

Just like for ordinary cohomology, it is possible to construct a sequence
\[ \cdots \longrightarrow \Hord{n-1}(A) \longrightarrow \Hord{n}(B,A)
    \longrightarrow \Hord{n}(B) \longrightarrow \Hord{n}(A) \longrightarrow
    \Hord{n+1}(B,A) \longrightarrow \cdots \]
This sequence will not turn out to be exact, but it does satisfy a related
property. The maps $\Hord{n}(B,A) \to \Hord{n}(B)$ are induced by the
inclusions $\Zord{n}(B,A) \to \Zord{n}(B)$, and the maps $\Hord{n}(B) \to
\Hord{n}(A)$ by $p$. The connecting homomorphism $\partial : \Hord{n}(A) \to
\Hord{n+1}(B,A)$ is manufactured as follows. Take any $x \in \Hord{n}(A)$ and
represent it by $a \in \Zord{n}(A)$. By surjectivity of $p$, there exists a $b
\in B^n$ for which $p(b) = a$. Then $\delta(b)$ is an element of
$\Zord{n+1}(B,A)$, because $\delta(\delta(b)) = 0$ and $p(\delta(b)) =
\delta(p(b)) = \delta(a) \geq 0$, where we used that $a \in \Zord{n}(A) =
\precone(\delta)$. Let $\partial(x)$ be the cohomology class of $\delta(b)$ in
$\Hord{n+1}(B,A)$. This does not depend on the choice of $b$, since if both
$p(b)$ and $p(b')$ are equal to $a$, then $c := b'-b$ satisfies $\delta(b') -
\delta(b) = \delta(c)$ and $p(c) = 0$, so $\delta(b) \sim \delta(b')$.

An exact sequence is a sequence in which the image of each morphism is the
kernel of the next one. In accordance with our general theme of replacing
kernels with precones, we wish to show that in order cohomology the image of
each morphism is the precone of the next one. Observe that the cohomology
monoids are not ordered in general, so it is not immediately clear what the
precone of a map between them should be. However, there is always a pre-order
on $\Hord{n}(A)$, defined in the following way:
let $a,b \in A^n$, and let $[a], [b]$ be the corresponding cohomology classes.
We say that $[a] \leq [b]$ if and only if there exists $c \in A^{n-1}$ such
that $a + \delta(c) \leq b$ in $A^n$.

\begin{lemma}
    The relation $\leq$ is a well-defined pre-order on $\Hord{n}(A)$.
\end{lemma}
\begin{proof}
    Suppose that $a \sim a'$ and $b \sim b'$, and that $a + \delta(c) \leq b$.
    Then there are $a''$ and $b''$ such that $a-a' = \delta(a'')$ and $b-b' =
    \delta(b'')$. Let $c' = a'' - b'' + c$, then
    \[ a' + \delta(c') = a' + a - a' - b + b' + \delta(c) \leq b - b + b' = b'. \]
    Hence the order does not depend on the choice of representatives. It is
    clear that $\leq$ is reflexive and transitive.
\end{proof}

Likewise, on the relative cohomology monoid $\Hord{n}(B,A)$ we define $[a]
\leq [b]$ if and only if there exists $c \in B^{n-1}$ such that $a + \delta(c)
\leq b$ and $p(c) = 0$.

\begin{proposition}
    \label{prop:PreconeExactSequence}
    In the sequence $\Hord{n}(B) \overset{p}{\longrightarrow} \Hord{n}(A)
    \overset{\partial}{\longrightarrow} \Hord{n+1}(B,A)$, we have
    $\precone(\partial) = \im(p)$.
\end{proposition}
\begin{proof}
    Suppose that $x \in \precone(\partial)$. Represent it by $a \in
    \Zord{n}(A)$, then there exists $b \in B^n$ such that $\delta(b)$ is
    positive in cohomology, and $p(b) = a$. Positivity in cohomology means
    that there exists $c$ such that $\delta(b) \geq \delta(c)$ and $p(c) = 0$.
    Define $d = b - c$, then $d$ lies in $\Zord{n}(B)$ because $\delta(b) \geq
    \delta(c)$. Furthermore $p(d) = p(b) - p(c) = a$, hence $x = [a] \in
    \im(p)$.

    Conversely, take $x \in \im(p)$ and represent $x$ by $a \in \Zord{n}(A)$.
    Then $a = p(b)$ for some $b \in \Zord{n}(B)$. It suffices to show that
    $[\delta(b)] \geq 0$. Since $b \in \Zord{n}(B)$, we have $\delta(b) \geq
    0$, therefore $[\delta(b)] \geq 0$.
\end{proof}

Similarly one can prove that $\precone(p) = \im(i)$.
Unfortunately it is not the case in general that $\precone(i) =
\im(\partial)$, but we will only need the property from
Proposition~\ref{prop:PreconeExactSequence}.

We will now specialize the homological algebra theory above to obtain order
cohomology of an effect algebra.
Let $E$ be an effect algebra, and let $A$ be an ordered abelian group. We
wish to define order cohomology of $E$ with coefficients in $A$. Often our
coefficient group will be $\mathbb{R}$.

Define the abelian group $C^n(E;A) = A^{\tests_n(E)}$ of maps from
$(n+1)$-tests on $E$ to $A$. To avoid cluttered notation, we will often
suppress the coefficient group $A$. The group $C^n(E)$ forms an ordered
abelian group with pointwise positive cone $C^n(E;A)\positive =
(A\positive)^{\tests_n(E)}$. We will construct a cochain complex out of the
groups
\[ \mathcal{C}^n(E;A) = \Disc(C^n(E;A)) \times C^{n-1}(E;A). \]
Each $\mathcal{C}^n(E)$ is an ordered abelian group with positive cone $\{ 0
\} \times C^{n-1}(E;A)\positive$. 

The groups $C^n(E)$ already form a cochain complex with the usual coboundary
maps $\delta : C^n(E) \to C^{n+1}(E)$, given by an alternating sum over
boundary maps.
We make the groups $\mathcal{C}^n(E)$ into a cochain complex by defining
coboundaries
\[ \delta^{\mathcal{C}}(\varphi, \psi) = (\delta \varphi, \varphi - \delta
    \psi). \]
When no confusion is possible, we will write $\delta^{\mathcal{C}}$ simply as
$\delta$. From the fact that $\delta^2 = 0$ it easily follows that also
$(\delta^{\mathcal{C}})^2 = 0$, so this is indeed a cochain complex. The
resulting order cohomology monoids $\CHord{n}(E;A) = \precone(\delta) /
\im(\delta)$ are the cohomology of $E$ with coefficients in $A$. From now on
we will assume that our coefficient group $A$ is $\mathbb{R}$ and write
$\CHord{n}(E;\mathbb{R})$ as $\CHord{n}(E)$.

We will determine the order cohomology monoids of an effect algebra $E$ in low
degrees. We have $\mathcal{C}^0(E) = \Disc(C^0(E)) \cong \Disc(\mathbb{R})$.
For the cochain complex in degree 1, we will use that
$\tests_1(E)$ can be identified with $E$, by letting $(a_0, a_1) \in
\tests_1(E)$ correspond to $a_1 \in E$. Hence $\mathcal{C}^1(E) \cong
\Disc(\mathbb{R}^E) \oplus \mathbb{R}$. The coboundary map $\delta :
\mathcal{C}^0(E) \to \mathcal{C}^1(E)$ is given by
\[ \delta^0 : \Disc(\mathbb{R}) \to \Disc(\mathbb{R}^E) \oplus \mathbb{R},
    \quad r \mapsto (\delta(r),r) = (0,r) \]
The zeroth cohomology monoid is $\CHord{0}(E) = \precone(\delta^0) =
\mathbb{R}_{\geq 0}$.

We continue with the first cohomology monoid. For this we will identify
$\tests_2(E)$ with $\{ (a,b) \mid a,b \in E, a \eaplus b \text{ is defined}
\}$, again by letting a 3-test $(a,b,c)$ correspond to $(b,c)$. We have
$\mathcal{C}^2(E) \cong \Disc(\mathbb{R}^{\tests_2(E)}) \oplus \mathbb{R}^E$,
and the coboundary $\delta^1 : \mathcal{C}^1(E) \to \mathcal{C}^2(E)$
satisfies
\[ \delta^1(\varphi, r) = \left( ((a,b) \mapsto \varphi(b) - \varphi(a \eaplus b) +
    \varphi(a)), \varphi \right). \]
By definition of the positive cone on $\mathcal{C}^2(E)$, the precone of
$\delta^1$ consists of those pairs $(\varphi : E \to \mathbb{R}, r\in
\mathbb{R})$ for which $\varphi(b) - \varphi(a \eaplus b) + \varphi(a) = 0$
whenever $a \eaplus b$ is defined, and $\varphi \geq 0$. In other words,
an element of $\precone(\delta^1)$ is a map $E \to \mathbb{R}_{\geq 0}$ that
preserves addition, together with a real number. In cohomology, two of these
elements are identified whenever their difference is a coboundary, which
happens if and only if it is of the form $(0,r)$. Hence a pair $(\varphi,r)$
is equivalent to $(\psi, s)$ precisely when $\varphi = \psi$. Thus the second
component of the pair collapses in cohomology, i.e.
\[ \CHord{1}(E) \cong \{ \varphi : E \to \mathbb{R}_{\geq 0} \mid \varphi(a
    \eaplus b) = \varphi(a) + \varphi(b) \}. \]

In particular, any state on $E$ is a member of the first cohomology monoid, so
it is possible to perform a construction similar to the one in
Section~\ref{sec:Applications}. Assume that $E$ lies in a larger effect
algebra $F$, via an inclusion $E \hookrightarrow F$. We wish to know when a
state on $E$ can be extended to a state on $F$. The sequence for relative
cohomology obtained earlier gives a connecting homomorphism $\partial :
\CHord{1}(E) \to \CHord{2}(F,E)$. Since $\St(E) \subseteq \CHord{1}(E)$, the
connecting homomorphism can be applied to any state on $E$.

\begin{theorem}
    Let $i : E \hookrightarrow F$ be an injective morphism of effect
    algebras, and let $\sigma : E \to [0,1]$ be a state. The following are
    equivalent:
    \begin{enumerate}
        \item The state $\sigma$ extends to a state $\tau$ on $F$, for which
            $\tau \circ i = \sigma$.
        \item The state $\sigma$ lies in the precone of the connecting
            homomorphism $\partial : \CHord{1}(E) \to \CHord{2}(F,E)$.
    \end{enumerate}
\end{theorem}
\begin{proof}
    If $\sigma$ extends to a state on $F$, then $\sigma$ lies in the image of
    the restriction map $p : \tau \mapsto \tau \circ i$. By
    Proposition~\ref{prop:PreconeExactSequence}, $\sigma$ is an element of
    $\precone(\partial)$.

    Conversely, if $\sigma \in \precone(\partial)$, then by the same
    proposition, it is of the form $\tau \circ i$ for some $\tau \in
    \CHord{1}(F)$. It remains to be checked that $\tau$ is a state. Since
    $\tau$ lies in the first cohomology monoid, it is an additive map $F \to
    \mathbb{R}_{\geq 0}$. Furthermore $\tau(1) = \tau(i(1)) = \sigma(1) = 1$,
    since $\sigma$ is a state. For any $a \in F$, we have
    \[ \tau(a) + \tau(a^\bot) = \tau(a \eaplus a^\bot) = 1, \]
    hence $\tau(a) \in [0,1]$ since $\tau$ maps into the positive reals. This
    proves that $\tau$ is an additive map $F \to [0,1]$ preserving $1$, in
    other words, a state.
\end{proof}

We conclude that order cohomology of effect algebras provides a method to
check whether states on an effect algebra extend to states on a larger effect
algebra, without any false positives.

\begin{example}
    The Bell state $\sigma : E_A \otimes E_B \to [0,1]$ is not classically
    realizable, in the sense that it does not factor through any power set.
    Therefore, for any set $X$, the state $\sigma$ does not lie in
    $\precone(\partial : \CHord{1}(E_A \otimes E_B) \to \CHord{2}(\Pow(X),E_A
    \otimes E_B))$.

    On the other hand, the Bell state is quantum realizable. This means that
    there exists a Hilbert space $H$ such that $\sigma$ factors through the
    projection lattice $\Proj(H)$. Observe that $\Proj(H)$ is an effect
    algebra because it is an orthomodular lattice. The above theorem tells us
    that $\sigma \in \precone(\partial : \CHord{1}(E_A \otimes E_B) \to
    \CHord{2}( \Proj(H),E_A \otimes E_B))$.
\end{example}

\paragraph{Acknowledgements.}
The author is grateful to Ieke Moerdijk for his guidance during this research
project. Furthermore, thanks are due to Pieter Hofstra, Bert Lindenhovius,
Philip Scott, Sander Uijlen, and Bram Westerbaan for helpful discussions.
This research has been financially supported by the Netherlands Organisation
for Scientific Research (NWO) under TOP-GO grant no.\ 613.001.013 (The logic
of composite quantum systems).


\bibliographystyle{eptcs}
\bibliography{effalgcohom}

\end{document}